\documentclass{umd-thesis}
\usepackage{amsfonts, amssymb, amsmath, amsthm, amscd, graphics} 
\usepackage[dvips]{epsfig}
\usepackage[all]{xy}

\newcommand{\C}{\mathbb{C}}
\newcommand{\tr}[1]{\mathrm{tr}(#1)}
\newcommand{\xb}{\mathbf{x}}
\newcommand{\yb}{\mathbf{y}}
\newcommand{\zb}{\mathbf{z}}
\newcommand{\pol}{\mathrm{pol}}
\newcommand{\ub}{\mathbf{u}}
\newcommand{\vb}{\mathbf{v}}
\newcommand{\wb}{\mathbf{w}}
\newcommand{\dg}[1]{|\!|#1|\!|}
\newcommand{\tdg}[1]{|\!|#1|\!|_{\mathrm{tr}}}
\newcommand{\aq}{/\!\!/}
\newcommand{\X}{\frak{X}}
\newcommand{\R}{\frak{R}}
\newcommand{\hm}{\mathrm{Hom}}
\newcommand{\SL}{\mathrm{SL}(3,\C)}
\newcommand{\ot}{\mathrm{Out}}
\newcommand{\F}{\mathtt{F}}
\newcommand{\G}{\frak{G}}
\newcommand{\xt}{\mathtt{x}}
\newcommand{\yt}{\mathtt{y}}
\newcommand{\wt}{\mathtt{w}}
\newcommand{\bt}{\mathtt{b}}
\newcommand{\rt}{\mathtt{r}}
\newcommand{\Z}{Z_\rt}
\newcommand{\M}{\mathbf{M}}
\newcommand{\id}{\mathbb{I}}
\newcommand{\D}{\frak{D}}
\newcommand{\ti}[1]{t_{(#1)}}
\newcommand{\dP}[1]{\frac{\partial P}{\partial \ti{#1}}}
\newcommand{\dQ}[1]{\frac{\partial Q}{\partial \ti{#1}}}
\newcommand{\Ad}{\mathrm{Ad}}
\newcommand{\gAd}{\frak{g}_{\Ad_\rho}}
\newcommand{\Cn}[1]{C^{#1}(\F_r;\frak{g}_{\mathrm{Ad}_\rho})}

\newtheorem{theorem}{Theorem}[section]
\newtheorem{lemma}[theorem]{Lemma}
\newtheorem{corollary}[theorem]{Corollary}
\newtheorem{prop}[theorem]{Proposition}
\newtheorem{fact}[theorem]{Fact}

\numberwithin{equation}{section}

\title{$\mathrm{SL}(3,\C)$-Character Varieties and $\mathbb{RP}^2$-Structures\\ on a Trinion}

\author{Sean Lawton}

\date{2006}

\advisor{Professor William M. Goldman}
\committee{Professor Dieter R. Brill\\
Professor John J. Millson\\
Professor Serguei Novikov\\
Professor James A. Schafer}

\dedication{I dedicate this work to my daughter.  If mathematics was my \emph{raison d'etre} prior to December $19$, $2004$, thereafter it was
Jaeda Sabine Lawton.}

\abstractfile{Denote the free group on two letters by $\F_2$ and the $\SL$-representation variety of $\F_2$ by $\R=\hm(\F_2,\SL)$.
There is a $\SL$-action on the coordinate ring of $\R$, and the geometric points of the subring of invariants is an affine
variety $\X$.  We determine explicit minimal generators and defining relations for the subring of invariants and show $\X$ is a hyper-surface
in $\C^9$.  Our choice of generators exhibit $\ot(\F_2)$ symmetries which allow for a succinct expression of the defining relations.  
We then show $\C[\X]$ is a Poisson algebra with respect to a presentation of $\F_2$ imposed by a punctured surface.  We work out the bracket on all generators when the 
surface is a thrice punctured sphere, or a trinion.  The moduli space of convex real projective structures on a trinion, denoted by $\frak{P}$, is a subset of $\X$.  
Lastly, we determine explicit conditions in terms of $\C[\X]$ that distinguish this moduli space.}

\acknowledgements{The author thanks William Goldman for overall guidance, substantial suggestions, and edits that always improved the quality of earlier drafts of this
paper.  Bill has not only been a mentor and friend, he is a model for me as an academic.  The author thanks Richard Schwartz, John Millson, and the 
University of Maryland's VIGRE-NSF program for generously supporting this work. In particular, he thanks John Millson for suggesting an analysis of the singular locus, 
which has improved the quality of this paper.  The author has benefited from fruitful conversations with Ben Howard, Elisha Peterson, Adam Sikora, and Joseph Previte, 
and thanks them for their time and insight.  He especially thanks Joseph Previte and Eugene Xia for generously sharing their calculations with him.  Additionally, he 
thanks his family for unending support.  In particular, without my wife Deborah, I would be nowhere.}

\begin{document}

\makefrontmatter

\chapter{Introduction}
The purpose of this paper is to present a self-contained description of a minimal generating set and defining relations for the ring of invariants
$$\C[\X]=\C[\SL\times\SL]^{\SL}.$$  This generating set exhibits symmetries which allow for an explicit and succinct expression of the invariant ring
as a quotient.

Explicit minimal generators were first found by \cite{Du} in $1935$, and later by \cite{SR, MS, T} and graphically by \cite{Si}.  The much more
general results of \cite{AP} additionally provide minimal generators.  However, \cite{N}, and later \cite{ADS} were the
first to explicitly describe the defining relations.  In an unpublished calculation \cite{PX} independently describe the defining relations as well.  
For the state-of-the-art, see \cite{DF}.  Our treatment provides the most succinct and transparent description by uncovering symmetries which provide 
a framework for generalization.

Thereafter, we show $\C[\X]$ is a Poisson algebra and demonstrate the power of our description of $\C[\X]$ by computing the bracket.  Finally, we apply these results to understand 
and describe moduli of convex real projective structures on a sphere with three disks removed.  These applications strongly use results in \cite{G1,G2,G5,Ki}.

\section{Algebraic Structure of $\SL$}

The group $\SL$ has the structure of an algebraic set since it is the zero set of the polynomial $$D=\det\left(
\begin{array}{ccc}
x_{11} & x_{12} & x_{13}\\
x_{21} & x_{22} & x_{23}\\
x_{31} & x_{32} & x_{33}\\
\end{array}\right)-1$$ on $\mathbb{C}^9$.  Here $x_{ij}\in
\C[x_{11},x_{12},x_{13},x_{21},x_{22},x_{23},x_{31},x_{32},x_{33}]$, the polynomial ring over $\C$ in
$9$ indeterminates.  As such denote $\SL$ by $\G$.  The coordinate ring of $\G$ is given by $$\mathbb{C}[\G]=\mathbb{C}[x_{ij}\ | \
1\leq i,j \leq 3]/(D).$$  Since $D$ is irreducible, $(D)$ is a prime ideal.  So the algebraic set $\G$ is in fact an affine variety.

The polynomial $D$ is irreducible, since the determinant of a $2\times 2$ matrix is irreducible by inspection, and $$D=x_{11}\det\left(\begin{array}{cc}   
x_{22} & x_{23}\\
x_{32} & x_{33} \end{array}\right) -x_{21}\det\left(\begin{array}{cc}     
x_{12} & x_{13}\\
x_{32} & x_{33} \end{array}\right) +x_{31}\det\left(\begin{array}{cc}      
x_{12} & x_{13}\\ 
x_{22} & x_{23} \end{array}\right)-1.$$  Thus $D$ cannot be factored in $x_{11},x_{21},x_{31}$ since it is of degree $1$ in those variables, but cannot be factored 
otherwise since the coefficients of those variables are irreducible as well.

Moreover, if the partial derivatives of $D$ are all zero, then $D=-1$ and so the Jacobian ideal has no solutions.  Hence, $\G$ is a non-singular algebraic variety.

\section{Representation $\&$ Character Varieties of a Free Group}\label{repchar}

Let $\F_r$ be the free group of rank $r$ generated by $\{\xt_1,...,\xt_r\}.$  The map $$\hm(\F_r,\G)\longrightarrow
\G^{\times r}$$ defined by sending $$\rho \mapsto (\rho(\xt_1),\rho(\xt_2),...,\rho(\xt_r))$$ is a
bijection.  Since $\G^{\times r}$ is the $r$-fold product of irreducible algebraic sets, $\G^{\times r} \cong \hm(\F_r,\G)$ is an affine
variety.  Moreover, since the product of smooth varieties over $\C$ is smooth, $\hm(\F_r,\G)$ is non-singular.

As such $\hm(\F_r, \G)$ is denoted by $\R$ and referred to as the $\SL$-{\it representation variety of $\F_r$}.

Let $\C[\R]$ be the coordinate ring of $\R$.  Our preceding remarks imply
$\C[\R]\cong \C[\G]^{\otimes r}$.
For $1\leq k\leq r,$ define  a {\it matrix variable} of the complex polynomial ring in $9r$ indeterminates by
$$\xb_k=\left(
\begin{array}{ccc}
x^k_{11} & x^k_{12} & x^k_{13}\\
x^k_{21} & x^k_{22} & x^k_{23}\\
x^k_{31} & x^k_{32} & x^k_{33}\\
\end{array}\right).$$
Let $\Delta$ be the ideal $(\det(\xb_k)-1\ |\ 1\leq k\leq r)$ in $\C[\R]$.  Then
$$\C[\R]=\mathbb{C}[x^k_{ij}\ | \ 1\leq i,j \leq 3,\ 1\leq k\leq r]/\Delta.$$

Let $(\xb_1,\xb_2,...,\xb_r)$ be an $r$-tuple of matrix variables.  An element $f\in \C[\R]$ is a function defined in
terms of such $r$-tuples.  There is a $\G$-action on $\C[\R]$ given by diagonal conjugation.  That is, for $g\in \G$ $$g\cdot
f(\xb_1,\xb_2,...,\xb_r)= f(g^{-1}\xb_1 g,...,g^{-1}\xb_r g).$$
The subring of invariants of this action $\C[\R]^{\G}$ is a finitely
generated $\C$-algebra (see \cite{D,P1,R}).  Consequently, the {\it character variety}
$$\X=\mathrm{Spec}_{max}(\C[\R]^{\G})$$ is the irreducible algebraic set whose coordinate ring is the
ring of invariants.  Therefore, $\C[\X]$ includes all polynomial maps of the form $\tr{\xb_{i_1}\xb_{i_2}\cdots\xb_{i_k}},$ where $1\leq i_j\leq r$.
For $r>1$, the Krull dimension of $\X$ is $8r-8$ since generic elements have zero dimensional isotropy (see \cite{D}, page $98$).  

There is a regular map $\R\stackrel{\pi}{\to}\X$ which factors through $\R/\G$:  let $\frak{m}$ be a maximal ideal
corresponding to a point in $\R$, then the composite isomorphism $\C\to \C[\R]\to \C[\R]/\frak{m}$ implies that the composite map $\C\to
\C[\R]^\G\to \C[\R]^\G/(\frak{m}\cap \C[\R]^\G)$ is an isomorphism as well.  Hence the contraction $\frak{m}\cap \C[\R]^\G$ is maximal, and since for any $g\in \G$, 
$\left(g\frak{m}g^{-1}\right)\cap \C[\R]^\G=\frak{m}\cap \C[\R]^\G$, $\pi$ factors through $\R/\G$ (see \cite{E}, page $38$).  
Although $\R/\G$ is not generally an algebraic set, $\X$ is the categorical quotient $\R\aq\G$, and since $\G$ is a (geometrically) reductive algebraic group $\pi$ is surjective 
and maps closed $\G$-orbits to points (see \cite{D}).

\subsubsection{Completely Reducible Representations}
For every representation $\rho\in \R$, $\C^3$ is a $\F_r$-module induced by $\rho$.  A {\it completely reducible} representation is one that is a direct sum of irreducible 
subrepresentations.  Such representations induce a {\it semi-simple} module structure on $\C^3$, and irreducible representations respectively result in {\it simple} modules.  
For any composition series of the $\F_r$-module associated to $\rho$, $\C^3=V_0\supset V_1\supset\cdots \supset V_{l}=0$, there is a semi-simple $\F_r$-module $W=\bigoplus 
V_{i}/V_{i+1}$.  With respect to a chosen basis of $W$, there exists a completely reducible representation $\rho^{(s)}$.  However, its conjugacy class is independent of any 
basis and moreover the Jordon-H\"{o}lder theorem implies that this class is also independent of the composition series.  

We characterize these representations by their orbits.  If $\rho$ is not completely reducible ($l>1$), 
then it is reducible and so for $\wt\in\F_r$ has the form:  $$\left[\begin{array}{ccc}a(\wt)
& b(\wt)& c(\wt)\\0 &d(\wt) &e(\wt)\\0&f(\wt)&g(\wt)\end{array}\right].$$  In this form, conjugating by $$\left[\begin{array}{ccc}1 & 0& 0\\0 &1/n
&0\\0&0&1/n\end{array}\right],$$ and taking the limit as $n\to\infty$ results in $$\left[\begin{array}{ccc}a(\wt)           
& 0& 0\\0 &d(\wt) &e(\wt)\\0&f(\wt)&g(\wt)\end{array}\right].$$  This limiting representation is $\rho^{(s)}$, if it had two irreducible summands.  Otherwise we may conjugate 
$\rho$ so $f(\wt)$ may be taken to be $0$.  Then conjugating this form of $\rho$ by $$\left[\begin{array}{ccc}1 & 0& 0\\0 &1/n
&0\\0&0&1/n^2\end{array}\right]$$ and taking the limit as $n\to\infty$ results in $$\left[\begin{array}{ccc}a(\wt)           
& 0& 0\\0 &d(\wt) &0\\0&0&g(\wt)\end{array}\right],$$ which is $\rho^{(s)}$ when it has three irreducible summands.  Either way, we have a sequence, $\rho_n\in \G\rho$, beginning 
at $\rho$ and limiting to $\rho^{(s)}\notin \G\rho$.  It follows that if $\G\rho$ is closed then $\rho$ is completely reducible.  

For the converse, we first show that $\pi(\rho)=\pi(\psi)$ if and only if $\rho^{(s)}=\psi^{(s)}$.  Indeed, suppose that $\pi(\rho)=\pi(\psi)$.  Then their characteristic 
polynomials are equal:  $\chi_{\rho}=\chi_{\psi}$.  Thus $\chi_{\rho^{(s)}}=\chi_{\psi^{(s)}}$.  However semi-simple representations are determined by their characteristic 
polynomials, so $\rho^{(s)}=\psi^{(s)}$.  On the other hand, if $\rho^{(s)}=\psi^{(s)}$ then $\rho_n\rightarrow\rho^{(s)}=\psi^{(s)}\leftarrow\psi_n$.  This in turn implies 
$\overline{\G\rho}\cap\overline{\G\psi}$ is not empty, and so $\pi(\rho)=\pi(\psi)$.
  
Now suppose $\rho$ has a non-closed orbit, and let $\psi$ be an element of $\overline{\G\rho}-\G\rho$.  Then $\rho$ and $\psi$ are not conjugate.  Without loss of generality, we 
can assume that $\G\psi$ is closed since the dimension of each subsequent sub-orbit decreases.  So $\psi=\psi^{(s)}$ and $\pi(\psi)=\pi(\rho)$.  Hence, $\rho^{(s)}=\psi$ and so 
$\rho$ cannot be completely reducible else it would be conjugate to $\psi$, which it is not.  In other words, if $\rho$ is completely reducible, then $\G\rho$ is closed.

Let $\R^{ss}$ be the subset of $\R$ containing only completely reducible representations.  Then we have just shown that $\R^{ss}/\G$ is in bijective correspondence (as sets) 
to $\X$, and the following diagram commutes:

\begin{displaymath}
\begin{CD}
\R@>>>\X\\
@VVV      @AAA\\
\R^{ss} @>>> \R^{ss}/G.\\
\end{CD}
\end{displaymath}

For a complete treatment of the above arguments see \cite{A,P2}.

\subsubsection{Simple Representations}

Let $\R^s\subset \R^{ss}$ be the set of irreducible representations, and let $\R^{reg}$ be the regular points in $\R$; that is the representations that have closed orbits
and have minimal dimensional isotropy.  These points form an open dense subset of $\R$ (see \cite{D}).

We claim that $\R^{reg}=\R^s$, if $\F_r$ has rank greater than $1$.  We have already seen that the irreducible representations have closed orbits, since they are 
completely reducible.  So it remains to show that $\rho$ is irreducible if and only if its isotropy has minimal dimension.  First, however, we address the case of $\F_1$.

In this case, all representations have an invariant subspace since the characteristic polynomial always has a root over $\C$.  So there are no irreducible representations, and 
the semi-simple representations are exactly the diagonalizable matrices.  Moreover, the dimension of the isotropy of any representation is at least $2$-dimensional since any 
matrix commutes with itself.  Also the set of matrices with distinct eigenvalues is dense; and any diagonalizable matrix has a repeated eigenvalue if and only if its 
isotropy has dimension greater than $2$.  So in this case, $\R^{reg}$ is the set of matrices with distinct eigenvalues, and $\R^s=\{\emptyset\}$.

Otherwise, $\F_r$ has rank at least $2$.  If $\rho\in \R^{ss}$ has an invariant subspace, it has non-zero dimensional isotropy since it fixes at least one line in $\C^3$.  On 
the other hand, the representations that have at least two distinct matrix variables having no shared eigenspaces have isotropy equal to the center, which is generated by the cubic 
roots of unity and so is zero-dimensional.  If a representation does not have this property then it must be reducible.  Hence, the minimal dimension of isotopy is zero
which is realized if and only if $\rho \in \R^s\subset \R^{ss}$.  Thus when $\F_r$ has $r>1$, then $\R^{reg}=\R^{s}$.  

In \cite{A}, it is shown that $\R^s\aq \G$ is a smooth irreducible variety.  Moreover, in \cite{G2} it is shown that $\G$ acts properly on $\R^s$, and although the action is 
not effective, the kernel is the center $\mathbb{Z}_3$.  Thus the induced ``infinitesimal'' action on the tangent space is in fact effective, since the tangent map 
corresponding to the center is zero.  Thus, if $\rho\in\R^s$, the tangent space to an orbit, $T_{\rho}(\mathcal{O}_{\rho})$, is isomorphic to $\frak{g}$, the Lie 
algebra of $\G$.  Together with properness, this implies that $\R^s\to\R\aq \G$ is a local submersion which in turn implies 
$T_{\rho}(\R\aq \G)\cong T_{\rho}(\R)/T_{\rho}(\mathcal{O}_{\rho})$ whenever $\rho$ is irreducible.  

It is not always the case that the tangent space to the quotient is the quotient of tangent spaces.  See \cite{Ka} for example.  The issue that arises is that there can be smooth 
points in the quotient that have positive-dimensional isotropy.  At these points, $T_{\rho}(\R\aq \G)\not\cong T_{\rho}(\R)/T_{\rho}(\mathcal{O}_{\rho}),$ seen by simply comparing 
dimensions.  On the other hand, it is not clear what happens at the singular points in the quotient since they will necessarily have positive-dimensional isotropy but yet their 
Zariski tangent space will also jump in dimension.  A dimension count is not sufficient however, since it may be the case that the differential to the projection at such a point 
is not surjective.

\chapter{Polynomial Matrix Identities}

Let $\F^{+}_r$ be the free monoid generated by $\{\xt_1,...,\xt_r\}$, and let $\M^+_r$ be the monoid
generated by $\{\xb_1,\xb_2,...,\xb_r\},$ as defined in Chapter $\ref{repchar}$, under matrix multiplication and with identity element $\id$ the $3\times 3$ identity matrix.
There is a surjection $\F^{+}_r\to \M^+_r$, defined by mapping $\xt_i\mapsto \xb_i$.
Let $\mathbf{w} \in \M^+_r$ be the image of $\mathtt{w}\in \F^{+}_r$ under this map.  Further, let $|\cdot|$ be the function
that takes a reduced word in $\F_r$ to its word length.  Then by \cite{P1,R}, we know $\C[\X]$ is not only finitely generated, but in fact
generated by
\begin{equation}\label{generators}
\{\tr{\mathbf{w}}\ |\ \mathtt{w}\in \F^{+}_r, \ |\mathtt{w}|\leq 7\}.
\end{equation}

Let $\xb^*_k$ be the transpose of the matrix of cofactors of $\xb_k$.  In other words, the $(i,j)^{\text{th}}$ entry of
$\xb^*_k$ is $$(-1)^{i+j}\mathrm{Cof}_{ji}(\xb_k);$$ that is, the determinant obtained by
removing the $j^{\text{th}}$ row and $i^{\text{th}}$ column of $\xb_k$.  Let $\M_r^*$ be the monoid generated by $\{\xb_1,\xb_2,...,\xb_r\}$ and 
$\{\xb^*_1,\xb^*_2,...,\xb^*_r\}.$  

Observe that $(\xb\yb)^*=\yb^*\xb^*$ for all $\xb,\yb \in \M^+_r$, and $\xb\xb^*=\det(\xb)\id.$  Now let
$\mathbf{N}_r$ be the normal sub-monoid generated by $$\{\det(\xb_k)\id\ |\ 1\leq k\leq r\},$$ and
subsequently define $\M_r=\M_r^*/\mathbf{N}_r$.  Notice in $\M_r$, $\xb^*=\xb^{-1}$, and thus $\M_r$ is a group.

We will need the structure of an algebra, and to that end let $\C \M_r$ be the group algebra defined over $\C$ with respect to matrix addition
and scalar multiplication in $\M_r$.  Likewise, let $\C\M^*_r$ be the semi-group algebra of the monoid $\M^*_r$.

The following commutative diagram relates these objects:
\begin{displaymath}
\begin{CD}
\F_r^+ @>>>\F_r   @= \F_r\\
@VVV  @.         @VVV\\
\M_r^+ @>>> \M_r^* @>>> \M_r\\
@VVV    @VVV       @VVV\\
\C \M_r^+ @>>>\C \M_r^* @>>> \C \M_r @>\mathrm{tr}>> \C[\X].\\
\end{CD}
\end{displaymath}
Since the trace is non-degenerate, all relations in $\C[\X]$ arise from relations in $\C\M_r$.
\section{Relations}
The Cayley-Hamilton theorem applies to this context and so for any $\xb\in \C \M_r$,
\begin{align} \label{cayham}
\xb^3-\tr{\xb}\xb^2+ \tr{ \xb^* } \xb-\det(\xb)\id=0.
\end{align}   

By direct calculation, or by Newton's trace formulas
\begin{align}
\tr{ \mathbf{x}^* }&=\frac{1}{2}\left(\tr{ \mathbf{x} }^2-\tr{ \mathbf{x}^2 }\right)\label{trinv}.
\end{align}

Together \eqref{cayham} and \eqref{trinv} imply
\begin{align}
\det(\mathbf{x})&=\frac{1}{3}\tr{\xb^3}+\frac{1}{6}\tr{\xb}^3-\frac{1}{2}\tr{\xb}\tr{\xb^2}.\label{dettr}
\end{align}

If $\xb, \yb \in \M_r$ then multiplying equation $\eqref{cayham}$ on the right by $\xb^{-1}\yb$ yields,
\begin{eqnarray}\label{cayham2}
\xb^2\yb-\tr{\xb}\xb \yb +\tr{\xb^{-1}}\yb-\xb^{-1}\yb=0.
\end{eqnarray}
Computations similar to those that follow may be found in \cite{MS, SR}.  For any $\xb, \yb \in \C \M_r$ and any $\lambda \in \C$, equation
$\eqref{cayham}$ implies
\begin{align}\label{cayhamsum}
(\xb+\lambda \yb)^3-\tr{\xb+\lambda \yb}(\xb+\lambda \yb)^2 + \tr{(\xb+\lambda
\yb)^*}(\xb+\lambda\yb)-\det(\xb+\lambda \yb)\id=0.
\end{align}

Using equations $\eqref{cayham}$, $\eqref{trinv}$, and $\eqref{dettr}$, we derive equations
\begin{align}\label{detsum}
\det(\xb+\lambda \yb)=&\lambda^3\left(\frac{1}{3}\tr{\yb^3} +\frac{1}{6}\tr{\yb}^3-\frac{1}{2}\tr{\yb}\tr{\yb^2}\right)+\nonumber\\
&{}\lambda^2\left(\tr{\xb\yb^2}+\frac{1}{2}\tr{\xb}\tr{\yb}^2-\frac{1}{2}\tr{\xb}\tr{\yb^2} -\tr{\yb}\tr{\xb\yb}\right)+\nonumber\\
&{}\lambda^1\left(\tr{\xb^2\yb}+\frac{1}{2}\tr{\yb}\tr{\xb}^2 -\frac{1}{2}\tr{\yb}\tr{\xb^2}-\tr{\xb}\tr{\xb\yb}\right)+\nonumber\\
&{}\lambda^0\left(\frac{1}{3}\tr{\xb^3} +\frac{1}{6}\tr{\xb}^3-\frac{1}{2}\tr{\xb}\tr{\xb^2}\right),
\end{align} and
\begin{align}\label{eq:1}
(\xb+\lambda\yb)\tr{(\xb+\lambda\yb)^*}=
&\lambda^3\left(\frac{1}{2}\tr{\yb}^2\yb-\frac{1}{2}\tr{\yb^2}\yb\right)+\nonumber\\
&{}\lambda^2\left(\frac{1}{2}\tr{\yb}^2\xb-\frac{1}{2}\tr{\yb^2}\xb +\tr{\xb}\tr{\yb}\yb-\tr{\xb\yb}\yb\right)+\nonumber\\
&{}\lambda^1\left(\frac{1}{2}\tr{\xb}^2\yb-\frac{1}{2}\tr{\xb^2}\yb+\tr{\xb}\tr{\yb}\xb-\tr{\xb\yb}\xb\right)+\nonumber\\
&{}\lambda^0\left(\frac{1}{2}\tr{\xb}^2\xb-\frac{1}{2}\tr{\xb^2}\xb\right).
\end{align} 

Substituting equations $\eqref{detsum}$ and $\eqref{eq:1}$ into $\eqref{cayhamsum}$ produces equation
\begin{align}
0=\lambda^3\bigg(\yb^3-\tr{\yb}\yb^2+\frac{1}{2}\tr{\yb}^2&\yb-\frac{1}{2}\tr{\yb^2}\yb-\nonumber\\  
\frac{1}{3}\tr{\yb^3}\id-&\frac{1}{6}\tr{\yb}^3\id+\frac{1}{2}\tr{\yb}\tr{\yb^2}\id\bigg)+\nonumber\\
\lambda^2\bigg(\xb\yb^2+\yb^2\xb+\yb\xb\yb-\tr{\xb}&\yb^2-\tr{\yb}\xb\yb-\tr{\yb}\yb\xb+\nonumber\\
\frac{1}{2}\tr{\yb}^2\xb-\frac{1}{2}\tr{\yb^2}&\xb+ \tr{\xb}\tr{\yb}\yb-\tr{\xb\yb}\yb -\tr{\xb\yb^2}\id-\nonumber\\
\frac{1}{2}&\tr{\xb}\tr{\yb}^2\id+\frac{1}{2}\tr{\xb}\tr{\yb^2}\id +\tr{\yb}\tr{\xb\yb}\id\bigg)+\nonumber\\
\lambda\bigg(\yb\xb^2+\xb^2\yb+\xb\yb\xb-\tr{\yb}\xb^2-&\tr{\xb}\yb\xb-\tr{\xb}\xb\yb+\nonumber\\
\frac{1}{2}\tr{\xb}^2\yb-\frac{1}{2}\tr{\xb^2}\yb+&\tr{\xb}\tr{\yb}\xb-\tr{\xb\yb}\xb-\tr{\yb\xb^2}\id-\nonumber\\
\frac{1}{2}&\tr{\yb}\tr{\xb}^2\id+\frac{1}{2}\tr{\yb}\tr{\xb^2}\id+\tr{\xb}\tr{\xb\yb}\id\bigg) +\nonumber\\
\lambda^0\bigg(\xb^3-\tr{\xb}\xb^2+\frac{1}{2}\tr{\xb}^2\xb&-\frac{1}{2}\tr{\xb^2}\xb-\nonumber\\
\frac{1}{3}\tr{\xb^3}\id-&\frac{1}{6}\tr{\xb}^3\id+\frac{1}{2}\tr{\xb}\tr{\xb^2}\id\bigg).
\end{align}

However, $\C \M_r [\lambda]$ does not have zero divisors, so each coefficient of a power of $\lambda$ is zero.  In particular,
\begin{align}\label{polarization}
&\yb\xb^2 +\xb^2\yb+\xb\yb\xb= \tr{\yb}\xb^2+\tr{\xb}\yb\xb +\tr{\xb}\xb\yb
-\tr{\xb}\tr{\yb}\xb +\tr{\xb\yb}\xb+  \nonumber\\
&\tr{\yb\xb^2}\id+\tr{\xb}\tr{\xb\yb}\id -\frac{1}{2}\left(\tr{\xb}^2\yb -\tr{\xb^2}\yb- \tr{\yb}\tr{\xb}^2\id +\tr{\yb}\tr{\xb^2}\id\right).
\end{align}

Define $\pol(\xb,\yb)$ to be the right hand side of equation $\eqref{polarization}$; that is,
\begin{align}\pol(\xb,\yb)=\yb\xb^2 +\xb^2\yb+\xb\yb\xb.\label{polarization:2}\end{align}

Then substituting $\xb$ by the sum $\xb + \zb$ in equation $\eqref{polarization:2}$, yields the fundamental expression
\begin{align}\label{fundamental}
\xb\zb\yb+\zb\xb\yb+\yb\xb\zb+\yb\zb\xb+\xb\yb\zb+\zb\yb\xb=\pol(\xb+\zb,\yb)-\pol(\xb,\yb)-\pol(\zb,\yb).
\end{align}

Taking the trace of equation $\eqref{polarization}$ after multiplying it on the left by $\ub$ and on the right by $\vb$ yields equation
\begin{align}\label{fund1}
\tr{\ub\yb\xb^2\vb}+&\tr{\ub\xb^2\yb\vb}=-\tr{\ub\xb\yb\xb\vb}+\tr{\yb}\tr{\ub\xb^2\vb}+\tr{\xb}\tr{\ub\yb\xb\vb}+\nonumber\\ 
&\tr{\xb}\tr{\ub\xb\yb\vb}-\left(\tr{\xb}\tr{\yb}-\tr{\xb\yb}\right)\tr{\ub\xb\vb}+\nonumber\\
&\left(\tr{\yb\xb^2}-\tr{\xb}\tr{\xb\yb}\right)\tr{\ub\vb}-\frac{1}{2}\left(\tr{\xb}^2-\tr{\xb^2}\right)\tr{\ub\yb\vb} +\nonumber\\
&\frac{1}{2}\left(\tr{\yb}\tr{\xb}^2 -\tr{\yb}\tr{\xb^2}\right)\tr{\ub\vb}.
\end{align}

Suppose $\xb,\yb\in \M_r$.  Then substituting $\vb=\xb^{-1}$ and $\ub=\yb^{-1}\xb^{-1}$ into equation $\eqref{fund1}$, provides equation
\begin{align}\label{fund2}
\tr{\xb\yb\xb^{-1}\yb^{-1}}=&-\tr{\yb\xb\yb^{-1}\xb^{-1}}-3+\tr{\yb}\tr{\yb^{-1}}+2\tr{\xb}\tr{\xb^{-1}}-\nonumber\\
&\tr{\xb}\tr{\yb}\tr{\xb^{-1}\yb^{-1}} +\tr{\xb\yb}\tr{\xb^{-1}\yb^{-1}}- \nonumber\\ 
&\tr{\xb^{-1}}\tr{\yb^{-1}\xb^{-1}\yb\xb^{-1}} +\big(\tr{\yb\xb^2}-\tr{\xb}\tr{\xb\yb}+\nonumber\\
&\tr{\xb^{-1}}\tr{\yb}\big)\tr{\yb^{-1}\xb^{-2}}.
\end{align}

Equation $\eqref{cayham2}$ implies equation
\begin{align}\label{eq:2}
\tr{\yb^{-1}\xb^{-2}}=&\tr{\xb^{-1}}\tr{\xb^{-1}\yb^{-1}}-\tr{\xb}\tr{\yb^{-1}}+\tr{\xb\yb^{-1}}.
\end{align}
Since $\tr{\yb^{-1}\xb^{-1}\yb\xb^{-1}}=\tr{\left(\xb^{-1}\yb^{-1}\right)^2\yb^2},$
\begin{align}\label{eq:3}
\tr{\yb^{-1}\xb^{-1}\yb\xb^{-1}}=&\tr{\xb^{-1}\yb^{-1}}\tr{\xb^{-1}\yb}-\tr{\xb}\tr{\yb}\tr{\yb^{-1}}+\\
&\tr{\yb}\tr{\xb\yb^{-1}}+\tr{\xb}+\tr{\xb\yb}\tr{\yb^{-1}}.\nonumber
\end{align}
Substituting equations $\eqref{eq:2}$ and $\eqref{eq:3}$ into equation $\eqref{fund2}$, we then derive the fundamental commutator relation
\begin{align}\label{polyp1}
\tr{\xb\yb\xb^{-1}\yb^{-1}}=& -\tr{\yb\xb\yb^{-1}\xb^{-1}}+\tr{\xb}\tr{\xb^{-1}}\tr{\yb}\tr{\yb^{-1}}
+\tr{\xb}\tr{\xb^{-1}} +\nonumber\\
&\tr{\yb}\tr{\yb^{-1}}+\tr{\xb\yb}\tr{\xb^{-1}\yb^{-1}}+\tr{\xb\yb^{-1}}\tr{\xb^{-1}\yb}-\nonumber\\
&\tr{\xb^{-1}}\tr{\yb}\tr{\xb\yb^{-1}}-\tr{\xb}\tr{\yb^{-1}}\tr{\xb^{-1}\yb} -\nonumber\\
&\tr{\xb}\tr{\yb}\tr{\xb^{-1}\yb^{-1}}-\tr{\xb\yb}\tr{\xb^{-1}}\tr{\yb^{-1}}-3.
\end{align}

\section{Generators}
From $\eqref{generators}$, we need only consider words in $\F^+_r$ of length $7$ or less.  The length of a   
reduced word is defined to be the number of letters, counting multiplicity, in the word.  We now define
the {\it weighted length}, denoted by $|\cdot|_w$, to be the number of letters of a reduced word having
positive exponent plus twice the number of letters having negative exponent, again counting multiplicity.

For example, in $\F_2$, we have $|\xt_1\xt_2|=|\xt_1\xt_2|_w=2$ but $|\xt_1^3\xt_2^{-2}|=3+2=5$ while
$|\xt_1^3\xt_2^{-2}|_w=3+2\cdot 2=7$.

For a polynomial expression $e$ in matrix variables with coefficients in $\C[\X]$, we define the {\it degree of $e$}, denoted by $|\!|e|\!|$, to be 
the largest weighted length of monomial words in the expression of $e$ that is minimal among all such expressions for $e$.  Additionally, we define 
the {\it trace degree of $e$}, denoted by $|\!|e|\!|_{\mathrm{tr}}$, to be the maximal degree over all monomial words within a trace coefficient of 
$e$.

For example, when $\xb, \yb\in \M_r$, $|\!|\pol(\xb,\yb)|\!|\leq \mathrm{max}\{2|\!|\xb|\!|,|\!|\xb|\!|+|\!|\yb|\!|\},$ while 
$|\!|\pol(\xb,\yb)|\!|_{\mathrm{tr}}\leq
2|\!|\xb|\!|+|\!|\yb|\!|$.

We remark that given two such expressions $e_1$ and $e_2$, $$|\!|e_1e_2|\!|\leq|\!|e_1|\!|+|\!|e_2|\!|\ \text{ and }\
|\!|e_1e_2|\!|_{\mathrm{tr}}\leq\mathrm{max}\{|\!|e_1|\!|_{\mathrm{tr}},|\!|e_2|\!|_{\mathrm{tr}}\}.$$

We are now prepared to characterize the generators of $\C[\X]$.

\begin{lemma}
$\C[\R]^\G$ is generated by $\tr{\mathbf{w}}$ such that $\mathtt{w}\in\F_r$ is cyclicly reduced, $|\mathtt{w}|_w \leq 6$, and all
exponents of letters in $\mathtt{w}$ are $\pm 1$.
\end{lemma}

\begin{proof}
For $n\geq 2$, equations $\eqref{cayham}$ and $\eqref{cayham2}$ determine equation
\begin{align}\label{powerreduce}
\tr{\ub\xb^n\vb}=&\tr{\xb}\tr{\ub\xb^{n-1}\vb}-\tr{\xb^{-1}}\tr{\ub\xb^{n-2}\vb}+\tr{\ub\xb^{n-3}\vb},
\end{align}
which recursively reduces $\tr{\mathbf{w}}$ to a polynomial in traces of words having no letter with exponent
other than $\pm 1$.  If however $n\leq -2$ then we first apply equation \eqref{cayham2} and then use \eqref{powerreduce}.  Hence it follows
that $\mathtt{w}$ is cyclically
reduced, and all letters have exponents $\pm 1$.

Substituting $\xb\mapsto \yb$ and $\yb\mapsto \xb\zb$ in equation $\eqref{polarization:2}$, and multiplying the resulting expression on the left by 
$\xb$ gives
\begin{align}\label{eq:4}\xb^2\zb\yb^2=-(\xb\yb^2\xb)\zb-(\xb\yb\xb)\zb\yb+\xb\pol(\yb,\xb).\end{align}

Replacing $\yb\mapsto \yb^2$ in equation $\eqref{polarization:2}$ produces $$\yb^2\xb^2+\xb^2\yb^2+\xb\yb^2\xb=\pol(\xb,\yb^2),$$ which substituted 
into equation
$\eqref{eq:4}$ yields equation
\begin{align}\label{eq:5}\xb^2\zb\yb^2=(\yb^2\xb^2+\xb^2\yb^2-\pol(\xb,\yb^2))\zb+(\yb\xb^2+\xb^2\yb-\pol(\xb,\yb))\zb\yb+\xb\pol(\yb,\xb\zb).\end{align}

Now substituting $\xb\mapsto \yb$ and $\yb\mapsto \xb^2\zb$ in equation $\eqref{polarization:2}$ and multiplying 
$\zb\yb^2+\yb^2\zb+\yb\zb\yb=\pol(\yb,\zb)$ on the left by $\xb^2$
gives $$\xb^2\zb\yb^2+\yb^2\xb^2\zb+\yb\xb^2\zb\yb=\pol(\yb,\xb^2\zb),$$ and $$\xb^2\zb\yb^2+\xb^2\yb^2\zb+\xb^2\yb\zb\yb=\xb^2\pol(\yb,\zb),$$
which substituted into equation $\eqref{eq:5}$ results in
$$3\xb^2\zb\yb^2=\pol(\yb,\xb^2\zb)+\xb\pol(\yb,\xb\zb)-\pol(\xb,\yb^2)\zb-\pol(\xb,\yb)\zb\yb+\xb^2\pol(\yb,\zb)\label{eq:6}.$$

Thus, $$|\!|\xb^2\zb\yb^2|\!|\lneq 2|\!|\xb|\!|+|\!|\zb|\!|+2|\!|\yb|\!| \text{ and } \tdg{\xb^2\zb\yb^2}\leq 2|\!|\xb|\!|+|\!|\zb|\!|+2|\!|\yb|\!|.$$

For the remainder of the argument assume $\xb,\yb,\zb,\ub,\vb,\wb$ are of length $1$.  Replacing $\yb\mapsto \ub+\vb$ in equation $\eqref{eq:6}$ we 
deduce
$|\!|\xb^2\zb(\ub^2+\ub\vb+\vb\ub+\vb^2)|\!|\leq 4.$  This in turn implies $|\!|\xb^2\zb(\ub\vb+\vb\ub)|\!|\leq 4$ and so 
$|\!|\xb^2\zb\wb(\ub\vb+\vb\ub)|\!|\leq 5.$  In a
like manner, we have that both $\dg{\xb^2\zb(\wb\ub\vb+\vb\wb\ub)}\leq 5$ and $\dg{\xb^2\zb(\wb\vb+\vb\wb)\ub}\leq 5$.  Hence we conclude that
$$\dg{2\xb^2\zb\wb\ub\vb}=\dg{\xb^2\zb\wb(\ub\vb+\vb\ub)+\xb^2\zb(\wb\ub\vb+\vb\wb\ub)-\xb^2\zb(\wb\vb+\vb\wb)\ub}\leq 5,$$ and 
$$\tdg{2\xb^2\zb\wb\ub\vb}\leq 6.$$

Replacing $\xb\mapsto \xb +\yb$ in $\xb^2\zb\wb\ub\vb$ we come to the conclusion that $\dg{\xb\yb\zb\wb\ub\vb+\yb\xb\zb\wb\ub\vb}\leq 5$.  That is, 
permuting $\xb$ and $\yb$
introduces a factor of $-1$ and a polynomial term of lesser degree.  Slight variation in our analysis concludes the same result for any transposition 
of two letters in the word
$\xb\yb\zb\wb\ub\vb$.

Therefore, if $\sigma$ is a permutation of the letters $\xb,\yb,\zb,\ub,\vb,\wb$ then
$$\dg{\xb\yb\zb\ub\vb\wb+\mathrm{sgn}(\sigma)\sigma(\xb\yb\zb\ub\vb\wb)}\leq 5\ \text{while}\ 
\tdg{\xb\yb\zb\ub\vb\wb+\mathrm{sgn}(\sigma)\sigma(\xb\yb\zb\ub\vb\wb)}\leq 6.$$

Lastly, making the substitutions $\xb\mapsto \xb\yb$, $\yb\mapsto \zb\ub$, and $\zb\mapsto \vb\wb$ in the fundamental expression 
$\eqref{fundamental}$, we derive
\begin{align} \label{eq:7} {\xb\yb}\vb\wb{\zb\ub}+&\vb\wb{\xb\yb}{\zb\ub}+{\zb\ub}{\xb\yb}\vb\wb+{\zb\ub}\vb\wb{\xb\yb}
+{\xb\yb}{\zb\ub}\vb\wb+\vb\wb{\zb\ub}{\xb\yb}=\nonumber\\&\pol({\xb\yb}+\vb\wb,{\zb\ub})-\pol({\xb\yb},\zb\ub)-\pol(\vb\wb, {\zb\ub}).\end{align}

However, each word on the left hand side of equation $\eqref{eq:7}$ is an even permutation of the first, so

$$\dg{6{\xb\yb}\vb\wb{\zb\ub}}\leq 5\ \text{and}\ \tdg{6{\xb\yb}\vb\wb{\zb\ub}}\leq 6.$$
Hence, if $\wb$ is a word of length $7$ or more, then $\tdg{\tr{\wb}}\leq 6$.
Moreover, this process gives an iterative algorithm for reducing such an expression.
\end{proof}
As an immediate general consequence we have the following classification of generators.  It is not minimal, however.

\begin{corollary}\label{generatorform}
$\C[\X]$ is generated by traces of the form
\begin{align*}&\tr{\xb_i},\tr{\xb_i^{-1}}, \tr{\xb_i\xb_j}, \tr{\xb_i\xb_j\xb_k}, \tr{\xb_i\xb_j^{-1}}, \tr{\xb_i^{-1}\xb_j^{-1}},\\ 
& \tr{\xb_i\xb_j\xb_k^{-1}}, \tr{\xb_i\xb_j\xb_k\xb_l}, \tr{\xb_i\xb_j\xb_k\xb_l\xb_m}, \tr{\xb_i\xb_j\xb_k\xb_l^{-1}},\\ 
& \tr{\xb_i\xb_j^{-1}\xb_k^{-1}}, \tr{\xb_i^{-1}\xb_j^{-1}\xb_k^{-1}}, \tr{\xb_i\xb_j\xb_k^{-1}\xb_l^{-1}},\\ 
&\tr{\xb_i\xb_j^{-1}\xb_k\xb_l^{-1}}, \tr{\xb_i\xb_j\xb_k\xb_l\xb_m^{-1}},\tr{\xb_i\xb_j\xb_k\xb_l\xb_m\xb_n},\end{align*}
where the indices may not be distinct in a given generator.  
\end{corollary}

\chapter{Structure of $\C[\G^{\times 2}\aq \G]$}

\section{Minimal Generators}

As a consequence of Corollary $\ref{generatorform}$, we have
\begin{lemma}\label{gens}
$\C [\G\times \G]^\G$ is generated by
\begin{align*}
&\tr{\xb_1},\ \ \tr{\xb_2},\ \ \tr{\xb_1\xb_2},\ \ \tr{\xb_1\xb_2^{-1}},\ \ \tr{\xb_1^{-1}},\\
&\tr{\xb_2^{-1}},\ \ \tr{\xb_1^{-1}\xb_2^{-1}},\ \ \tr{\xb_1^{-1}\xb_2},\ \ \tr{\xb_1\xb_2\xb_1^{-1}\xb_2^{-1}}.
\end{align*}
\end{lemma}

\begin{proof}
The words of weighted length $1,2,3,4$ with exponents $\pm 1$ are unambiguously cyclically equivalent to one of $$\xt_1,
\xt_2, \xt_1^{-1}, \xt_2^{-1}, \xt_1\xt_2, \xt_1\xt_2^{-1}, \xt_2\xt_1^{-1}, \xt_1^{-1}\xt_2^{-1}, (\xt_1\xt_2)^2.$$  But equation $\eqref{cayham2}$ reduces the
latter most of these in terms of the others.  All words in two letters of length $5$ are cyclically equivalent to a
word with an exponent whose magnitude is greater than $1$, except $\xt_1\xt_2^{-1}\xt_1\xt_2$, and $\xt_2\xt_1^{-1}\xt_2\xt_1$.
Both are cyclically equivalent to $(\xt_i\xt_j)^2\xt_j^{-2}$, which in turn, by equation $\eqref{powerreduce}$, reduces to
expressions in the other variables.  The only words of weighted length $6$ and with 
exponents only $\pm 1$ are $\xt_1\xt_2 \xt_1^{-1}\xt_2^{-1}$, its inverse, and $(\xt_1\xt_2)^3$.  But the latter most of
these is reduced by equation $\eqref{cayham}$.  Lastly, letting $\xb=\xb_1$ and $\yb=\xb_2$ in equation
$\eqref{polyp1}$, we have
\begin{align}\label{polyp2}
\tr{\xb_2\xb_1\xb_2^{-1}\xb_1^{-1}}=& -\tr{\xb_1\xb_2\xb_1^{-1}\xb_2^{-1}}+\tr{\xb_1}\tr{\xb_1^{-1}}\tr{\xb_2}\tr{\xb_2^{-1}}+\nonumber\\
&\tr{\xb_1}\tr{\xb_1^{-1}}+ \tr{\xb_2}\tr{\xb_2^{-1}}+\tr{\xb_1\xb_2}\tr{\xb_1^{-1}\xb_2^{-1}}+\nonumber\\
&\tr{\xb_1\xb_2^{-1}}\tr{\xb_1^{-1}\xb_2}-\tr{\xb_1^{-1}}\tr{\xb_2}\tr{\xb_1\xb_2^{-1}}-\nonumber\\
&\tr{\xb_1}\tr{\xb_2^{-1}}\tr{\xb_1^{-1}\xb_2} -\tr{\xb_1}\tr{\xb_2}\tr{\xb_1^{-1}\xb_2^{-1}}-\nonumber\\
&\tr{\xb_1\xb_2}\tr{\xb_1^{-1}}\tr{\xb_2^{-1}}-3,
\end{align}
which expresses the trace of the inverse of the commutator in terms of the other expressions.
\end{proof}

The center of $\G$ is $\zeta(\G)=\{\omega\id\ |\ \omega^3=1\}\cong\mathbb{Z}_3$.  There is an action of $\zeta(\G)^{\times 2}$ on $\C[\X]$ given by
$$(\omega_1\id,\omega_2\id)\cdot\tr{\mathbf{w}(\xb_1,\xb_2)}=\tr{\mathbf{w}(\omega_1\xb_1,\omega_2\xb_2)}= 
\omega_1^{|\mathbf{w}(\xb_1,\id)|_w}\omega_2^{|\mathbf{w}(\id,\xb_2)|_w}\tr{\mathbf{w}(\xb_1,\xb_2)}.$$  Applying this action to the generators and recording the
orbit by a $9$-tuple, we can distingish all generators and by doing so grade the ring:

\begin{prop}
$$\C[\X]=\sum_{(\omega_1,\omega_2)\in\mathbb{Z}_3\times \mathbb{Z}_3} \mathrm{P}_{(\omega_1,\omega_2)}$$ is a graded ring, where $\mathrm{P}_{(\omega_1,\omega_2)}$ is 
the linear span over $\C$ of all monomials whose orbit under $\mathbb{Z}_3\times \mathbb{Z}_3$ equals one of the orbits of $t_{(\pm i)}$; which themselves are in 
bijective correspondence with the elements of $\mathbb{Z}_3\times \mathbb{Z}_3$.
\end{prop}

In fact the situation is general.  For a rank $r$ free group, $\mathbb{Z}_3^{\times r}$ acts on the generators of $\C[\X]$ and gives a filtration.  However, 
since the relations are polarizations of the Cayley-Hamilton polynomial, which itself has a zero grading, no relation can compromise summands.  So the filtration is a
grading.

\section{Hyper-Surface in $\C^9$}

Let $$\overline{R}=\C[t_{(1)},t_{(-1)},t_{(2)},t_{(-2)},t_{(3)},t_{(-3)},t_{(4)},t_{(-4)},
t_{(5)},t_{(-5)}]$$ be the complex polynomial ring freely generated by $\{t_{(\pm i)},\ 1\leq i\leq 5\},$ and let
$$R=\C[t_{(1)},t_{(-1)},t_{(2)},t_{(-2)},t_{(3)},t_{(-3)},t_{(4)},t_{(-4)}]$$ be its subring generated
by $\{t_{(\pm i)},\  1\leq i\leq 4\},$  so
$\overline{R}=R[t_{(5)},t_{(-5)}].$
Define the following ring homomorphism,
$$R[t_{(5)},t_{(-5)}]\stackrel{ \Pi}{\longrightarrow} \C[\G\times \G]^\G$$ by

\begin{center}
\begin{tabular}{ll}
$t_{(1)}\mapsto\tr{\mathbf{x}_1}$ & $t_{(-1)}\mapsto\tr{\mathbf{x}_1^{-1}}$\\
$t_{(2)}\mapsto\tr{\mathbf{x}_2}$& $t_{(-2)}\mapsto\tr{\mathbf{x}_2^{-1}}$\\
$t_{(3)}\mapsto\tr{\mathbf{x}_1\mathbf{x}_2}$& $t_{(-3)}\mapsto\tr{\mathbf{x}_1^{-1}\mathbf{x}_2^{-1}}$\\
$t_{(4)}\mapsto\tr{\mathbf{x}_1\mathbf{x}_2^{-1}}$& $t_{(-4)}\mapsto\tr{\mathbf{x}_1^{-1}\mathbf{x}_2}$\\
$t_{(5)}\mapsto\tr{\mathbf{x}_1\mathbf{x}_2\mathbf{x}_1^{-1}\mathbf{x}_2^{-1}}$& $t_{(-5)}\mapsto\tr{\mathbf{x}_2\mathbf{x}_1\mathbf{x}_2^{-1}\mathbf{x}_1^{-1}}$.
\end{tabular}
\end{center}

It follows from Lemma \ref{gens} that $$\C[\X]\cong R[t_{(5)}, t_{(-5)}]/\ker(\Pi).$$ In other words, $\Pi$ is a
surjective algebra morphism.\\
We define $$P=t_{(1)}t_{(-1)}t_{(2)}t_{(-2)}-t_{(1)}t_{(2)}t_{(-3)}-t_{(-1)}t_{(-2)}t_{(3)}
-t_{(1)}t_{(-2)}t_{(-4)}-t_{(-1)}t_{(2)}t_{(4)}$$
$$+t_{(1)}t_{(-1)}+t_{(2)}t_{(-2)}+t_{(3)}t_{(-3)}+t_{(4)}t_{(-4)}-3,$$ and so $P\in R$.
Moreover, by equation $\eqref{polyp2}$, $$P-(t_{(5)}+t_{(-5)}) \in \ker(\Pi).$$

Hence it follows that the composite map $$R[t_{(5)}]\hookrightarrow R[t_{(5)}, t_{(-5)}]\twoheadrightarrow R[t_{(5)},
t_{(-5)}]/\ker(\Pi),$$ is an epimorphism.  Let $I$ be the kernel of this composite map, and suppose there exists $Q\in
R$ so $Q-t_{(5)}t_{(-5)}\in \ker(\Pi)$ as well.

Then under this assumption, we prove
\begin{lemma}\label{ideallemma}
$I$ is principally generated by the polynomial
\begin{equation}\label{idealequ}
t_{(5)}^2-Pt_{(5)}+Q.
\end{equation}
\end{lemma}
\begin{proof}
The following argument is an adaptation of one found in \cite{N}.

Certainly, $t_{(5)}^2-Pt_{(5)}+Q\in I$ for it maps into $R[t_{(5)},t_{(-5)}]/\ker(\Pi)$ to the coset
representative $t_{(5)}^2-(t_{(5)}+t_{(-5)})t_{(5)}+t_{(5)}t_{(-5)}=0$.   

On the other hand, observe $$R[t_{(5)}]/I\cong R[t_{(5)}, t_{(-5)}]/\ker(\Pi)\cong \C[\X],$$ the
dimension of $\X$ is $8$, and $R[t_{(5)}]/I$ has at most $9$ generators.   Then it must be the case that $I$ is principally generated since $R[\ti{5}]$ is a 
U.F.D., and thus a co-dimension $1$ irreducible subvariety of $\C^9$ must be given by one equation (see \cite{S} page $69$).  Moreover, $I$ is 
non-zero since otherwise the resulting dimension would necessarily be too large.

Seeking a contradiction, suppose there exists a polynomial identity comprised of only elements of $R$.  Then Krull's dimension theorem (see page $68$ in \cite{S}) implies
$t_{(5)}$ is free.  In other words, given any specialization of the generators of $R$, $t_{(5)}$ is not determined.  Consider 
$(\mathrm{SL}(2,\C)\times \{1\})^2\subset \G^2$; that is, the matrices of the form $\left(
\begin{array}{ccc}
a & b & 0\\
c & d & 0\\
0& 0 & 1\\
\end{array}\right)$ so $ad-bc=1$.  Then by restricting to pairs of such matrices, we deduce that $$\tr{\xb_1\xb_2\xb_1^{-1}\xb_2^{-1}}=
\tr{\xb_2\xb_1\xb_2^{-1}\xb_1^{-1}},$$ since for all $g\in \mathrm{SL}(2,\C),$ $\tr{g}=\tr{g^{-1}}$ (see \cite{LP}).
Then equation $\eqref{polyp2}$ becomes $$t_{(5)}=P/2,$$ which is decidedly not free of the
generators of $R$.  Thus, the generators of $R$ are algebraically independent in $R[t_{(5)}]/I$.

Since $I$ is principal and contains a monic quadratic over $R$, its generator is expression $\eqref{idealequ}$, or a
factor thereof.  We have just shown that there are no degree zero relations, with respect to $t_{(5)}$.  However,
if $I$ is generated by a linear polynomial over $R$ then $t_{(5)}$ is determined by the generators of
$R$.  However this in turn would imply that all
representations who agree by evaluation in $R$ also agree by evaluation under $t_{(5)}$.

Consider the representations 

\begin{tabular}{ccc}
$\F_2 \stackrel{\rho_1}{\longrightarrow} \G$ & &$\F_2 \stackrel{\rho_2}{\longrightarrow} \G$\\
$\xt_1 \longmapsto
\left(
\begin{array}{ccc}
a & 0 & 0\\
0 & b & 0\\
0& 0 & 1/ab\\
\end{array}\right)$ &\ \text{and}\  &
$\xt_1 \longmapsto
\left(
\begin{array}{ccc}
a & 0 & 0\\
0 & b & 0\\

0& 0 & 1/ab\\
\end{array}\right)$\\
$\xt_2 \longmapsto
\frac{1}{4^{1/3}}\left(
\begin{array}{ccc}
1 & 1 & -1\\
1 & -1 & 1\\
-1& -1 & -1\\
\end{array}\right)$
& & $\xt_2 \longmapsto
\frac{1}{4^{1/3}}\left(
\begin{array}{ccc}
1 &-1 & 1\\
-1 & -1 & -1\\
1& 1 & -1\\
\end{array}\right).$
\end{tabular}
\\ \\
It is a direct calculation to verify that they agree upon evaluation in $R$ but disagree under $t_{(5)}$.
\end{proof}

Lemmas \ref{gens} and \ref{ideallemma} together imply the following theorem whose result, in part, is given as an 
example of the powerful graphical techniques
developed in \cite{Si} and was shown to be true before that by \cite{T}.

\begin{theorem}\label{ranktwo}
$\G^{\times 2}\aq\G$ is isomorphic to an affine degree $6$ hyper-surface in $\C^9$, which maps onto $\C^8$.
\end{theorem}

\begin{proof}
Once we explicitly determine $Q$, it having degree $6$ will be apparent.
It then remains to show that $\X\rightarrow \C^8$ is a surjection.  To this end, let $(z_1-\zeta_1,...,z_8-\zeta_8)$ be a maximal ideal in the coordinate ring of 
$\C^8$.
Moreover, let $\zeta_9$ be defined to be a solution to $t^2-P(\zeta_1,...,\zeta_8)t+Q(\zeta_1,...,\zeta_8)=0$.  Then
$(t_{(1)}-\zeta_1,t_{(-1)}-\zeta_2,...,t_{(-4)}-\zeta_8,t_{(5)}-\zeta_9)+I$ is a maximal
ideal in $\C[\X]$, and so all maximal ideals of $\C[\C^8]$ are images of such in $\C[\X]$.
\end{proof}

\section{Singular Locus of $\X$.}

The surjection $\X\to \C^8$ is generically $2$-to-$1$, that is there are exactly two solutions to $$t^2-P(\zeta_1,...,\zeta_8)t+Q(\zeta_1,...,\zeta_8)=0$$ for every point 
in $\C^8$ except where $P^2-4Q=0$.  In this case, $$0=(t_{(5)}+t_{(-5)})^2-4t_{(5)}t_{(-5)}=(t_{(5)}-t_{(-5)})^2$$ which implies $t_{(5)}=t_{(-5)}=P/2$.  In $\X$, on the 
other hand, $t_{(5)}=P/2$ implies that $P^2-4Q=0$.  Let $\frak{L}$ denote the locus of solutions to $P^2-4Q=0$ in $\X$, which is a closed subset of $\X$. 

It is readily observed that the $t_{(5)}$ partial derivative of $t_{(5)}^2-Pt_{(5)}+Q$ is zero if and only if $t_{(5)}=P/2$.  The singular set in $\X$, denoted by $\frak{J}$, 
is the closed subset cut out by the Jacobian ideal; that is, the ideal generated by the formal partial derivatives of $\ti{5}^2-P\it{5}+Q$.  Thus 
$\frak{J}\subset\frak{L}$.  In the proof of Lemma $\ref{ideallemma}$, we observed that $(\mathrm{SL}(2,\C)\times\{1\})^{\times 2}\aq \G \subset
\frak{L}$.  Additionally, since matrices of the form
$\left(
\begin{array}{ccc}
a & 0 & 0\\
0 & b & 0\\
0& 0 & 1/ab\\
\end{array}\right)$ commute, restricting to pairs of such matrices enforces the relation
$$\tr{\xb_1\xb_2\xb^{-1}_1\xb^{-1}_2}=3=\tr{\xb_2\xb_1\xb^{-1}_2\xb^{-1}_1}.$$  Let $(\C^*)^2$ denote the subset of such matrices in $\G$.
Consequently, $(\C^*)^4\aq \G\subset \frak{L}$ as well.  We claim both sets satisfy all the generators of the Jacobian ideal, and so are singular.

Explicitly, the Jacobian ideal is generated by the polynomials $-\ti{5}\dP{i}+\dQ{i}$ for $1\leq |i|\leq 4$ and $2\ti{5}-P$.  Using the formulas for $P$ and $Q$ 
(see Chapter \ref{detq}) we derive:
\begin{align*}
\dP{1}&=-\ti{-4} \ti{-2} + \ti{-1} - \ti{-3} \ti{2} + \ti{-2} \ti{-1} \ti{2}\\
\dP{2}&=\ti{-2} - \ti{-3} \ti{1} + \ti{-2} \ti{-1} \ti{1} - \ti{-1} \ti{4}\\
\dP{3}&= \ti{-3} - \ti{-2} \ti{-1}\\
\dP{4}&=\ti{-4} - \ti{-1} \ti{2}
\end{align*}

\begin{align*}
\dP{-4}&=-\ti{-2} \ti{1} + \ti{4}\\
\dP{-3}&=-\ti{1} \ti{2} + \ti{3}\\
\dP{-2}&=-\ti{-4} \ti{1} + \ti{2} + \ti{-1} \ti{1} \ti{2} -\ti{-1} \ti{3}\\
\dP{-1}&=\ti{1} + \ti{-2} \ti{1} \ti{2} - \ti{-2} \ti{3} - \ti{2} \ti{4}
\end{align*}

\begin{align*}
\dQ{1}=&3\ti{-4} \ti{-2} + \ti{-3} \ti{-2}^2 - 6 \ti{-1} - 
      \ti{-2}^3 \ti{-1} + 2 \ti{-4} \ti{-3} \ti{1} - \\
      &2 \ti{-4} \ti{-2} \ti{-1} \ti{1} + 3 \ti{1}^2 + 
      3 \ti{-3} \ti{2} - \ti{-4} \ti{-2}^2 \ti{2} + \\
      &\ti{-2} \ti{-1} \ti{2} - 2 \ti{-3} \ti{-1} \ti{1} \ti{2} + 
      2 \ti{-2} \ti{-1}^2 \ti{1} \ti{2} - 3 \ti{-2} \ti{1}^2 \ti{2} +\\ 
      &\ti{-4} \ti{2}^2 - \ti{-3} \ti{-2} \ti{2}^2 + 
      \ti{-2}^2 \ti{-1} \ti{2}^2 - \ti{-1} \ti{2}^3 +\\ 
      &\ti{-4}^2 \ti{3} + \ti{-3} \ti{-1} \ti{3} - 
      \ti{-2} \ti{-1}^2 \ti{3} + 2 \ti{-2} \ti{1} \ti{3} - \\
      &\ti{-4} \ti{-1} \ti{2} \ti{3} + 2 \ti{1} \ti{2}^2 \ti{3} - 
      2 \ti{2} \ti{3}^2 + \ti{-3}^2 \ti{4} + \ti{-4} \ti{-1} \ti{4} -\\ 
      &\ti{-3} \ti{-2} \ti{-1} \ti{4} + 2 \ti{-2}^2 \ti{1} \ti{4} - 
      \ti{-1}^2 \ti{2} \ti{4} + 2 \ti{1} \ti{2} \ti{4} - 3 \ti{3} \ti{4} - \\
      &\ti{-2} \ti{2} \ti{3} \ti{4} - 2 \ti{-2} \ti{4}^2
\end{align*}
 
\begin{align*}
\dQ{2}=&\ti{-4} \ti{-3}^2 - 6\ \ti{-2} - 2 \ti{-4}^2 \ti{-1} - 
      \ti{-4} \ti{-3} \ti{-2} \ti{-1} + \ti{-3} \ti{-1}^2 - \\
      &\ti{-2} \ti{-1}^3 + 3 \ti{-3} \ti{1} - 
      \ti{-4} \ti{-2}^2 \ti{1} + \ti{-2} \ti{-1} \ti{1} - \\
      &\ti{-3} \ti{-1} \ti{1}^2 + \ti{-2} \ti{-1}^2 \ti{1}^2 - 
      \ti{-2} \ti{1}^3 + 2 \ti{-4} \ti{-1}^2 \ti{2} + \\
      &2 \ti{-4} \ti{1} \ti{2} - 2 \ti{-3} \ti{-2} \ti{1} \ti{2} + 
      2 \ti{-2}^2 \ti{-1} \ti{1} \ti{2} + 3 \ti{2}^2 -\\ 
      &3 \ti{-1} \ti{1} \ti{2}^2 - 3 \ti{-4} \ti{3} + 
      \ti{-3} \ti{-2} \ti{3} - \ti{-2}^2 \ti{-1} \ti{3} -\\ 
      &\ti{-4} \ti{-1} \ti{1} \ti{3} + 2 \ti{-1} \ti{2} \ti{3} + 
      2 \ti{1}^2 \ti{2} \ti{3} - 2 \ti{1} \ti{3}^2 +\\ 
      &\ti{-4} \ti{-2} \ti{4} + 3 \ti{-1} \ti{4} - 
      \ti{-1}^2 \ti{1} \ti{4} + \ti{1}^2 \ti{4} + 2 \ti{-3} \ti{2} \ti{4} -\\ 
      &2 \ti{-2} \ti{-1} \ti{2} \ti{4} - \ti{-2} \ti{1} \ti{3} \ti{4} + 
      \ti{3} \ti{4}^2
\end{align*}

\begin{align*}
\dQ{3}=&-6 \ti{-3} + \ti{-4} \ti{-2}^2 + 
      3 \ti{-2} \ti{-1} + \ti{-4}^2 \ti{1} + \\
      &\ti{-3} \ti{-1} \ti{1} - \ti{-2} \ti{-1}^2 \ti{1} + 
      \ti{-2} \ti{1}^2 - 3 \ti{-4} \ti{2} + \ti{-3} \ti{-2} \ti{2} - \\
      &\ti{-2}^2 \ti{-1} \ti{2} - \ti{-4} \ti{-1} \ti{1} \ti{2} + 
      \ti{-1} \ti{2}^2 + \ti{1}^2 \ti{2}^2 + 2 \ti{-4} \ti{-1} \ti{3} - \\
      &4 \ti{1} \ti{2} \ti{3} + 3 \ti{3}^2 + \ti{-4} \ti{-3} \ti{4} + 
      \ti{-1}^2 \ti{4} - 3 \ti{1} \ti{4} - \ti{-2} \ti{1} \ti{2} \ti{4} + \\
      &2 \ti{-2} \ti{3} \ti{4} + \ti{2} \ti{4}^2
\end{align*} 

\begin{align*}
\dQ{4}=&-6 \ti{-4} - 3 \ti{-3} \ti{-2} + \ti{-2}^2 \ti{-1} + 
      \ti{-3}^2 \ti{1} + \ti{-4} \ti{-1} \ti{1} - \\
      &\ti{-3} \ti{-2} \ti{-1} \ti{1} + \ti{-2}^2 \ti{1}^2 + 
      \ti{-4} \ti{-2} \ti{2} + 3 \ti{-1} \ti{2} -\\ 
      &\ti{-1}^2 \ti{1} \ti{2} + \ti{1}^2 \ti{2} + \ti{-3} \ti{2}^2 - 
      \ti{-2} \ti{-1} \ti{2}^2 + \ti{-4} \ti{-3} \ti{3} + \\
      &\ti{-1}^2 \ti{3} - 3 \ti{1} \ti{3} - \ti{-2} \ti{1} \ti{2} \ti{3} + 
      \ti{-2} \ti{3}^2 + 2 \ti{-3} \ti{-1} \ti{4} -\\ 
      &4 \ti{-2} \ti{1} \ti{4} + 2 \ti{2} \ti{3} \ti{4} + 3 \ti{4}^2.
\end{align*}

Let $\frak{i}$ be the polynomial mapping that sends $\ti{i}\mapsto \ti{-i}$, that is induced by the automorphism of $\F_2$ which sends $\xt_i\mapsto \xt_i^{-1}$.  Then it is 
readily observed that $\dP{i}=\frak{i}\Big(\dP{-i}\Big)$.
By working out the other partials of $Q$ one observes this same symmetry.  So we express the other four partials of $Q$ by:
$$\dQ{-1}=\frak{i}\Big(\dQ{1}\Big),
\ \dQ{-2}=\frak{i}\Big(\dQ{2}\Big),\ \dQ{-3}=\frak{i}\Big(\dQ{3}\Big),\ \dQ{-4}=\frak{i}\Big(\dQ{4}\Big),$$
which may be verified with {\it Mathematica}, or by hand.

To show that $(\mathrm{SL}(2,\C)\times\{1\})^{\times 2}\aq \G$ is contained in the singular set of $\X$, we will find the following proposition useful.
\begin{prop}
$(\mathrm{SL}(2,\C)\times\{1\})^{\times 2}\aq \G$ is contained in the algebraic set cut out by the following equations:
\begin{align}
\ti{-i}=&\ti{i}\ \mathrm{for}\ 1\leq |i| \leq 4 \label{sl21}\\
\ti{4}=&\ti{1}\ti{2} - \ti{3} - \ti{1} - \ti{2} + 3 \label{sl22}\\
\ti{5}=&3 - 3\ti{1} + \ti{1}^2 - 3\ti{2} + \ti{1}\ti{2} + \ti{2}^2 - 3\ti{3} + \ti{1}\ti{3} \nonumber\\
&+ \ti{2}\ti{3} - \ti{1}\ti{2}\ti{3} + \ti{3}^2. \label{sl23}
\end{align}
\end{prop}
\begin{proof}
Let $\mathrm{SL}(2,\C)\times\{1\}$ be hereafter denoted $\mathrm{SL}_3(\mathrm{SL}(2,\C))$, and let $A\in\mathrm{SL}(2,\C)$ correspond to 
$\tilde{A}\in\mathrm{SL}_3(\mathrm{SL}(2,\C))$.  Since 
the normal form of such a matrix, conjugating in $\G$, may be achieved by 
restricting the conjugation action to $\mathrm{SL}_3(\mathrm{SL}(2,\C))$, we know that $\mathrm{SL}_3(\mathrm{SL}(2,\C))^{\times 2}\aq \G$ has the same algebraic dimension as
$\mathrm{SL}(2,\C)^{\times 2}\aq \mathrm{SL}(2,\C)$.  It is shown in \cite{LP} that this dimension is $3$.  So if the above equations are satisfied by 
$\mathrm{SL}_3(\mathrm{SL}(2,\C))$ then $\mathrm{SL}_3(\mathrm{SL}(2,\C))^{\times 2}\aq \G$ is not cut out by any further equations.

Since any matrix in $\mathrm{SL}(2,\C)$ is conjugate to its inverse, it follows that $$\tr{\tilde{A}^{-1}}=\tr{\widetilde{A^{-1}}}=\tr{A^{-1}}+1=\tr{A}+1=\tr{\tilde{A}}.$$  
Hence equations $\eqref{sl21}$ are satisfied.  Next, using the $\mathrm{SL}(2,\C)$ identity $$A^2-\tr{A}A+\id=0,$$ one easily derives $\tr{AB^{-1}}=\tr{A}\tr{B}-\tr{AB}.$  It 
then follows
\begin{align*}
\tr{\tilde{A}\tilde{B}^{-1}}&=\tr{\tilde{A}\widetilde{B^{-1}}}=\tr{\widetilde{AB^{-1}}}\\
&=\tr{AB^{-1}}+1=\tr{A}\tr{B}-\tr{AB}+1\\
&=(\tr{\tilde{A}}-1)(\tr{\tilde{B}}-1)-(\tr{\widetilde{AB}}-1)+1\\
&=\tr{\tilde{A}}\tr{\tilde{B}}-\tr{\tilde{A}\tilde{B}}-\tr{\tilde{A}}-\tr{\tilde{B}}+3.
\end{align*}  
Hence, equation $\eqref{sl22}$ is satisfied.  Lastly, substituting equations $\eqref{sl21}$ and equation $\eqref{sl22}$ into the identity $\ti{5}=P/2$ we conclude equation 
$\eqref{sl23}$.
\end{proof}

Now, substituting these equations directly into the generators of the Jacobian ideal results in all generators reducing identically to $0$.  Hence 
$\mathrm{SL}_3(\mathrm{SL}(2,\C))^{\times 2}\aq \G\subset \frak{J}$.

Returning our attention to $(\C^*)^{\times 4}$, we observe that diagonal matrices are already in normal form.  So we simply evaluate the generators of the Jacobian ideal at 
pairs of diagonal matrices in $\G\times\G$.  Doing so again results in all generators reducing identically to $0$, and allows us to conclude $(\C^*)^{\times 4}\aq 
\G\subset \frak{J}$.  Both of these computations, due to the large number of variables, was completed using {\it Mathematica}.

In fact these are both prototypical examples.  We have already seen that in general, if $\rho$ is singular, then its orbit has positive-dimensional isotropy.  In the case 
of a free group of rank $1$, there are no singular points in the quotient and so the identity, which has maximal isotropy, remains smooth.  Hence the converse is not 
generally true.  In the case of a free group of rank $2$, the situation is much better.

Let $\mathrm{GL}(2,\C)\times \C^*$ be the subset of $\G$ consisting of elements of the form $$\left(
\begin{array}{ccc}
a & b & 0\\  
c & d & 0\\
0& 0 & \frac{1}{ad-bc}\\   
\end{array}\right)$$ so $ad-bc\not=0$.

Using {\it Mathematica} we verify that $(\mathrm{GL}(2,\C)\times \C^*)^{\times 2}\aq \G$ is singular, after verifying that $\ti{5}=P/2$ on this subset.  However, any 
completely reducible representation  that is not irreducible is conjugate to an element in $(\mathrm{GL}(2,\C)\times \C^*)^{\times 2}$ since there must be a shared 
eigenvector with respect to its matrix variables.  Moreover, $(\mathrm{GL}(2,\C)\times \C^*)^{\times 2}\aq \G$ has the same algebraic dimension as 
$\mathrm{GL}(2,\C)^{\times 2}\aq \mathrm{GL}(2,\C)$, which is $5$ (see \cite{DF}).  Since the branching locus has dimension $7$ they are not equal, and in fact we have shown
that a representation in $\X$ is singular if and only if its orbit has positive-dimensional isotropy.  This further implies that the set $\mathcal{X}=\frak{X}-\frak{J}$ is 
connected, since the complement of the branching locus in $\X$ divides it into two isomorphic open subsets (sheets).  Since the 
singular set is contained in the branching set, each sheet is smooth.  Thus there is a path between the sheets going through $\frak{L}-\frak{J}$.

As a final note, we give an example of such an element (actually we give a $2$-dimensional parameterization of $\frak{L}-\frak{J}$):

\begin{tabular}{c}
$\F_2 \stackrel{\rho}{\longrightarrow} \G$ \\
$\xt_1 \longmapsto
\left(
\begin{array}{ccc}
a & 0 & 0\\
0 & a & 0\\
0& 0 & 1/a^2\\
\end{array}\right)$

$\xt_2 \longmapsto
\frac{c^{1/3}}{4^{1/3}}\left(
\begin{array}{ccc}
1 & 1 & -1\\
1 & -1 & 1\\
-1/c& -1/c & -1/c\\
\end{array}\right),$
\end{tabular}

so long as $a^3\not=1$ and $c\not=0$.  

In the case, $c=1$, this follows since $\rho$ is the limit as $a-b\to 0$ with respect to $\rho_1$ in the proof of Lemma \ref{ideallemma}.  Since the limit enforces  
$\rho_1-\rho_2\to 0$, it must be the case that $\ti{5}\to P/2$, so $\rho$ is in $\frak{L}$.  Calculating the Jacobian relations we determine that all partial derivatives are 
$0$ except for 
\begin{align*}
-\ti{5}\dP{1}+\dQ{1}&=-\frac{(-1+a^3)^3}{4a^4}\\
-\ti{5}\dP{-1}+\dQ{-1}&=\frac{(-1+a^3)^3}{4a^5},
\end{align*}
which are clearly not always $0$.  The formulas for the partials and the analysis carry over exactly for any non-zero $c$, which can be verified by direct calculation.

\section{Determining $Q$}\label{detq}
For the proofs of Lemma $\ref{ideallemma}$ and subsequently Theorem $\ref{ranktwo}$ to be complete, it only remains 
to establish that there exists $Q\in R$ so $Q-t_{(5)}t_{(-5)}\in
\ker(\Pi)$.
Before doing so, we state and prove the following technical fact, which may be found in \cite{N}.

\begin{fact}\label{fact}
Define a bilinear form on the vector space of $n\times n$ matrices over $\C$
by $$\mathbb{B}(A,B)=n\tr{AB}-\tr{A}\tr{B}.$$
Then given vectors $A_1,...,A_{n^2},B_1,...,B_{n^2},$ the $n^2\!\times\! n^2$ matrix $\Lambda=\bigg(\mathbb{B}(A_i,B_j)\bigg)$
is singular.
\end{fact}
\begin{proof}
Consider the co-vector $$v(\ \ )=\left[\begin{array}{c}\mathbb{B}(A_1,\ \ )\\\mathbb{B}(A_2,\ \ )\\
\vdots\\\mathbb{B}(A_{n^2},\ \ )\end{array}\right].$$  If $B_1,...,B_{n^2}$ are linearly dependent then so are
$v(B_1),v(B_2),...,v(B_{n^2})$, which implies the
columns of $\Lambda$ are linearly dependent.  Otherwise there exists coefficients, not all zero, so
$$c_1 B_1+c_2 B_2+\cdots+c_{n^2}B_{n^2}=\id,$$ which implies $$c_1 v(B_1)+c_2 v(B_2)+\cdots+c_{n^2}v(B_{n^2})=0$$ since the identity $\id$ is
in the kernel of $\mathbb{B}(A,\  )$.  So again the columns of $\Lambda$ are linearly dependent.  Either way, $\Lambda$ is
singular.
\end{proof}

\begin{lemma}\label{qlemma} There exists a polynomial $Q\in R$ so $Q-t_{(5)}t_{(-5)}\in \ker(\Pi)$, and in particular
\begin{align}\label{q}
Q=&9-6t_{(1)}t_{(-1)} -6t_{(2)}t_{(-2)} -6t_{(3)}t_{(-3)} -6t_{(4)}t_{(-4)}+t_{(1)}^3 +t_{(2)}^3 +t_{(3)}^3 +t_{(4)}^3\nonumber\\ 
&+t_{(-1)}^3 +t_{(-2)}^3 +t_{(-3)}^3 +t_{(-4)}^3 -3t_{(-4)}t_{(-3)}t_{(-1)} -3t_{(4)}t_{(3)}t_{(1)} -\nonumber\\
&3t_{(-4)}t_{(2)}t_{(3)} -3t_{(4)}t_{(-2)}t_{(-3)}+3t_{(-4)}t_{(-2)}t_{(1)} +3t_{(4)}t_{(2)}t_{(-1)}+ \nonumber\\
&3t_{(1)}t_{(2)}t_{(-3)} +3t_{(-1)}t_{(-2)}t_{(3)}+t_{(-2)}t_{(-1)}t_{(2)}t_{(1)}+t_{(-3)}t_{(-2)}t_{(3)}t_{(2)} +\nonumber\\
&t_{(-4)}t_{(-1)}t_{(4)}t_{(1)} +t_{(-4)}t_{(-2)}t_{(4)}t_{(2)} +t_{(-3)}t_{(-1)}t_{(3)}t_{(1)}+\nonumber \\
&t_{(-3)}t_{(-4)}t_{(3)}t_{(4)}+t_{(-4)}^2t_{(-3)}t_{(-2)}  +t_{(4)}^2t_{(3)}t_{(2)} +t_{(-1)}^2t_{(-2)}t_{(-4)} +t_{(1)}^2t_{(2)}t_{(4)}+\nonumber\\
&t_{(1)}t_{(-2)}^2t_{(-3)} +t_{(-1)}t_{(2)}^2t_{(3)} +t_{(-4)}t_{(-3)}t_{(1)}^2 +t_{(4)}t_{(3)}t_{(-1)}^2 +\nonumber\\
&t_{(-4)}t_{(2)}t_{(-3)}^2 +t_{(4)}t_{(-2)}t_{(3)}^2 +t_{(-1)}^2t_{(-3)}t_{(2)} +t_{(1)}^2t_{(3)}t_{(-2)} +\nonumber\\
&t_{(-4)}t_{(1)}t_{(2)}^2 +t_{(4)}t_{(-1)}t_{(-2)}^2+t_{(-4)}t_{(3)}t_{(-2)}^2 +t_{(4)}t_{(-3)}t_{(2)}^2 +\nonumber\\
&t_{(1)}t_{(3)}t_{(-4)}^2 +t_{(-1)}t_{(-3)}t_{(4)}^2 +t_{(-1)}t_{(-4)}t_{(3)}^2+t_{(1)}t_{(4)}t_{(-3)}^2-2t_{(-3)}^2t_{(-2)}t_{(-1)}-\nonumber\\
&2t_{(3)}^2t_{(2)}t_{(1)} -2t_{(-4)}^2t_{(-1)}t_{(2)} -2t_{(4)}^2t_{(1)}t_{(-2)}+t_{(-1)}^2t_{(-2)}^2t_{(-3)}+t_{(1)}^2t_{(2)}^2t_{(3)}+\nonumber\\
&t_{(-4)}t_{(-1)}^2t_{(2)}^2+t_{(4)}t_{(1)}^2t_{(-2)}^2-t_{(-4)}t_{(-2)}^2t_{(2)}t_{(1)} -t_{(4)}t_{(2)}^2t_{(-2)}t_{(-1)}-\nonumber\\
&t_{(-3)}t_{(1)}^2t_{(-1)}t_{(2)}-t_{(3)}t_{(-1)}^2t_{(1)}t_{(-2)}- t_{(-3)}t_{(2)}^2t_{(-2)}t_{(1)} -t_{(3)}t_{(-2)}^2t_{(2)}t_{(-1)}-\nonumber\\
&t_{(-4)}t_{(-2)}t_{(-1)}t_{(1)}^2 -t_{(4)} t_{(2)}t_{(1)}t_{(-1)}^2-t_{(-1)}t_{(-2)}^3t_{(1)}-t_{(-1)}t_{(2)}^3 t_{(1)} -\nonumber\\
&t_{(-1)}^3t_{(-2)}t_{(2)}-t_{(1)}^3t_{(-2)}t_{(2)}-t_{(-4)}t_{(-3)}t_{(-2)}t_{(-1)}t_{(2)}-t_{(4)}t_{(3)}t_{(2)}t_{(1)}t_{(-2)}-\nonumber\\ 
&t_{(-1)}t_{(1)}t_{(2)}t_{(-4)}t_{(3)} -t_{(-1)}t_{(1)}t_{(-2)}t_{(4)}t_{(-3)}+ t_{(-2)}t_{(-1)}^2t_{(1)}^2t_{(2)} +t_{(-1)}t_{(-2)}^2t_{(2)}^2t_{(1)}.
\end{align}
\end{lemma}

\begin{proof}
The following argument is an adaptation of an existence argument given in \cite{N}, which we use not only to show existence of $Q$, but to derive the 
explicit formulation of $Q$ as well.  Indeed, with respect to Fact $\ref{fact}$, let
\begin{align*}
&A_1=B_1=\xb_1\ \ \ A_4=B_4=\xb_2^{-1} \ \ \ A_7=B_7=\xb_1\xb_2^{-1}\\
&A_2=B_2=\xb_2 \ \ \ A_5=B_5=\xb_1\xb_2 \ \ A_8=B_8=\xb_2^{-1}\xb_1\\
&A_3=B_3=\xb_1^{-1} \ \ A_6=B_6=\xb_2\xb_1 \ \ A_9=B_9=\xb_2\xb_1^{-1}.
\end{align*}

Then we see that $\Lambda$ has exactly two entries with
$\tr{\mathbf{x}_1\mathbf{x}_2\mathbf{x}_1^{-1}\mathbf{x}_2^{-1}}$. After rewriting all matrix entries in terms of our
generators of $\C[\X]$, we have
$$0 =\det(\Lambda)=P_1\cdot\tr{\mathbf{x}_1\mathbf{x}_2\mathbf{x}_1^{-1}\mathbf{x}_2^{-1}}^2
+P_2\cdot\tr{\mathbf{x}_1\mathbf{x}_2\mathbf{x}_1^{-1}\mathbf{x}_2^{-1}}+P_3,$$
where $P_1,P_2,P_3$ are polynomials in terms of
$$\tilde{R}=\{\tr{\mathbf{x}_1}, \tr{\mathbf{x}_1^{-1}},
\tr{\mathbf{x}_2}, \tr{\mathbf{x}_2^{-1}}, \tr{\mathbf{x}_1\mathbf{x}_2},
\tr{\mathbf{x}_1^{-1}\mathbf{x}_2^{-1}},\tr{\mathbf{x}_1\mathbf{x}_2^{-1}},
\tr{\mathbf{x}_1^{-1}\mathbf{x}_2}\}.$$

If $P_1=0$ then we have a non-trivial relation among the elements of $\tilde{R}$, which we have already seen cannot exist.
Alternatively, one can specialize the elements of $\tilde{R}$ with the aid of a computer algebra system to verify that
$P_1\not=0$.  Then by direct calculation, using {\it Mathematica}, we verify that $P_2=-P \cdot P_1$.  Hence it follows that
$$-P_3=P_1(t_{(5)}^2-Pt_{(5)})=P_1(t_{(5)}^2-(t_{(5)}+t_{(-5)})t_{(5)})=-P_1t_{(5)}t_{(-5)},$$ and so we have shown the existence
of $$Q=t_{(5)}t_{(-5)}.$$  Lastly, we simplify $P_3/P_1,$ with the aid of {\it
Mathematica}, which turns out to be equation $\eqref{q}$.\end{proof}

\section{Outer Automorphisms}
Given any $\alpha\in \mathrm{Aut}(\F_2)$, we define $a_\alpha \in \mathrm{End}(\C[\X])$ by extending the following mapping
$$a_\alpha(\tr{\mathbf{w}})=\tr{\alpha(\mathbf{w})}.$$  If $\alpha \in \mathrm{Inn}(\F_2)$, then there exists $\mathtt{u}\in \F_2$ so for all
$\mathtt{w}\in \F_2$, $$\alpha(\mathtt{w})=\mathtt{uwu}^{-1},$$ which implies
$$a_\alpha(\tr{\mathbf{w})}=\tr{\ub\mathbf{w}\ub^{-1}}=\tr{\mathbf{w}}.$$  

Thus $\ot(\F_2)$ acts on $\C[\X]$.  By results in \cite{MKS}, $\ot(\F_2)$ is generated by the following mappings
\begin{align}
\tau&=\left\{
\begin{array}{l}
\xt_1\mapsto \xt_2\\
\xt_2\mapsto \xt_1
\end{array}\right.\\
\iota&=\left\{
\begin{array}{l}
\xt_1\mapsto \xt_1^{-1}\\
\xt_2\mapsto \xt_2
\end{array}\right.\\
\eta&=\left\{
\begin{array}{l}
\xt_1\mapsto \xt_1\xt_2\\
\xt_2\mapsto \xt_2
\end{array}\right.
\end{align}
Let $\frak{D}$ be the subgroup generated by $\tau$ and $\iota$, and let $\C\frak{D}$ be the corresponding group ring.
Then $\C[\X]$ is a $\C\frak{D}$-module.

\begin{lemma}\label{preservelemma}
The action of $\C \frak{D}$ preserves $R$, and $\frak{D}$ fixes $P$ and $Q$.
\end{lemma}
\begin{proof}
First we note that it suffices to check $\{\iota, \tau\}$ on $$\{t_{(\pm i)},\  1\leq i\leq 4\},$$ since the former
generates $\C\frak{D}$ and the latter generates $R$. Secondly we observe that both $\iota$ and $\tau$ are
idempotent.

Indeed, $\iota$ maps the generators of $R$ as follows:
\begin{eqnarray*}
t_{(1)}\mapsto t_{(-1)} \mapsto t_{(1)}\\
t_{(3)}\mapsto t_{(-4)}\mapsto t_{(3)}\\
t_{(2)}\mapsto t_{(2)}\\
t_{(-2)}\mapsto t_{(-2)}\\
t_{(4)}\mapsto t_{(-3)}\mapsto t_{(4)}.
\end{eqnarray*} 
Likewise, $\tau$ maps the generators of $R$ by:
\begin{eqnarray*} 
t_{(1)}\mapsto t_{(2)} \mapsto t_{(1)}\\
t_{(-1)}\mapsto t_{(-2)}\mapsto t_{(-1)}\\
t_{(3)}\mapsto t_{(3)}\\
t_{(-3)}\mapsto t_{(-3)}\\
t_{(4)}\mapsto t_{(-4)}\mapsto t_{(4)}.
\end{eqnarray*}   
Hence both map into $R$.
For the second part of the lemma, it suffices to observe $\iota(t_{(\pm 5)})=t_{(\mp 5)}=\tau(t_{(\pm 5)}),$ because in $\C[\X]$, $P=t_{(-5)}+t_{(5)}$ and
$Q=t_{(5)}t_{(-5)}.$ \end{proof}

Observing $\iota(t_{(5)})=\tau(t_{(5)})=t_{(-5)}=P-t_{(5)},$ it is apparent that $\frak{D}$ does not act as a permutation group on the
entire coordinate ring of $\X$.  However, when restricted to $R$ there is

\begin{theorem}
$\frak{D}$ restricted to $R$ is group isomorphic to the dihedral group, $D_4$, of order $8$.  Moreover, the algebraically independent
generators are characterized as those which $\frak{D}$ acts on as a permutation group.
\end{theorem}

\begin{proof}
Let $S=\mathrm{Sym}(\pm 1,\pm 2, \pm 3, \pm 4)$ be the symmetric group of all permutations on the eight letters $\pm i$ for $1\leq i\leq 4$.
Then we have worked out, in the proof of Lemma \ref{preservelemma}, that $\tau$ acts on the subscripts of $t_{(\pm i)}$ as the
permutation $(1,2)(-1,-2)(4,-4)$ and likewise, $\iota$ acts as the permutation $(1,-1)(3,-4)(-3,4).$  Since $\frak{D}$ is generated by these 
elements, we certainly have a well defined injection $\frak{D}\to S$.  

The Cayley table for $\frak{D}$ is:

\begin{table}[!h]
\begin{center}
{\footnotesize
\begin{tabular}{|c||c|c|c|c|c|c|c|c|}
\hline
&$id$ & $\iota$ & $\tau$ & $\iota\tau$ & $\tau\iota$ & $\tau\iota\tau$ & $\iota\tau\iota$ & $\tau\iota\tau\iota$\\
\hline
\hline
$id$ & $id$& $ \iota$ & $\tau$ &  $\iota \tau$ & $\tau \iota$ & $\tau \iota \tau$ & $\iota\tau\iota$ & $\tau\iota\tau\iota$ \\
\hline
$\iota$ &$\iota$ & $id$ & $\iota\tau$ & $\tau$ & $\iota\tau\iota$ & $\tau\iota\tau\iota$ &  $\tau\iota$  & $\tau \iota \tau$ \\
\hline
$\tau$ & $\tau$ & $\tau\iota$ & $id$ & $\tau\iota\tau$ & $\iota$ & $\iota\tau$ & $\tau\iota\tau\iota$ & $\iota \tau \iota$\\
\hline
$\iota\tau$ & $\iota\tau$ & $\iota\tau\iota$ & $\iota$ & $\tau\iota\tau\iota$ & $id$ & $\tau$ & $\tau\iota\tau$ &$\tau\iota$ \\
\hline
$\tau \iota$ &$\tau \iota$ & $\tau$ &$\tau \iota \tau$ & $id$ & $\tau\iota\tau\iota$ &$\iota\tau\iota$ & $\iota$ &$\iota\tau$\\
\hline
$\tau \iota \tau$ & $\tau \iota \tau$ &$\tau\iota\tau\iota$ &$\tau \iota$ &$\iota\tau\iota$ &$\tau$ & $id$ &$\iota\tau$ & $\iota$\\
\hline
$\iota \tau \iota$ & $\iota \tau \iota$ &$\iota\tau$ &$\tau\iota\tau\iota$ &$\iota$ & $\tau\iota\tau$ & $\tau\iota$ & $id$ & $\tau$\\
\hline
$\tau\iota\tau\iota$ &$\tau\iota\tau\iota$ & $\tau\iota\tau$ & $\iota\tau\iota$ & $\tau\iota$ &$\iota\tau$ & $\iota$ & $\tau$ &$id$\\
\hline
\end{tabular}}
\end{center}
\caption{Cayley Table of $\D$}
\end{table}
where
\begin{center}
\begin{tabular}{ll}
$id\mapsto (1)$ & $\iota\mapsto (1,-1)(3,-4)(-3,4)$\\
$\tau\mapsto (1,2)(-1,-2)(4,-4)$ & $\iota \tau\mapsto (1,2,-1,-2)(3,-4,-3,4)$\\
$\tau \iota\mapsto (1,-2,-1,2)(3,4,-3,-4)$ & $\tau \iota \tau\mapsto (2,-2)(3,4)(-3,-4)$\\
$\iota\tau\iota\mapsto (1,-2)(2,-1)(3,-3)$ & $\tau\iota\tau\iota\mapsto (1,-1)(2,-2)(3,-3)(4,-4)$.
\end{tabular}
\end{center}

It is an elementary exercise in group theory (see \cite{H}) to show any group presentable as $$\{a,b\ \big|\ |a|=n\geq 3,\
|b|=2,\ ba=a^{-1}b\}$$ is isomorphic to the dihedral group $D_n$ of order $2n$.  However, letting $a=\tau\iota$ and $b=\iota$ we
see $|a|=4$, $|b|=2$, $\frak{D}$ is generated by $a$ and $b$, and  $$ba=\iota\tau\iota=(\tau\iota)^{-1}\iota=a^{-1}b.$$

The last statement in the theorem follows from the fact that $\{t_{(\pm i)}\ |\ 1\leq i\leq 4\}$ are algebraically independent and $\frak{D}$ 
does not act as a permutation
group if $t_{(5)}$ were included.
\end{proof}

As already noted, the group ring $\C\frak{D}$ acts on $\C[\X]$.  By brute force computation, one can establish the following succinct
expressions for the polynomial relations $P$ and $Q$.

\begin{corollary}
In $\C\frak{D}$ define $\mathbb{S}_\frak{D}$ to be the group ``symmetrizer'' $$\sum_{\sigma\in \frak{D}}\sigma.$$  Then  
$P=\mathbb{S}_\frak{D}(p)-3$ and $Q=\mathbb{S}_\frak{D}(q)+9$ where $p$ and $q$ are given by:
\begin{align*}
p=\frac{1}{8}\big(t_{(1)}t_{(-1)}t_{(2)}t_{(-2)}-4t_{(1)}t_{(-2)}t_{(-4)}+2t_{(1)}t_{(-1)}+&2t_{(3)}t_{(-3)}\big)\\
q=\frac{1}{8}\big(2t_{(-2)}t_{(-1)}^2 t_{(1)}^2t_{(2)}+4t_{(1)}^2t_{(2)}^2t_{(3)}-4t_{(1)}^3t_{(-2)}t_{(2)}&-8t_{(-4)}t_{(-2)}t_{(-1)}t_{(1)}^2-\\
4t_{(4)}t_{(3)}t_{(2)}t_{(1)}t_{(-2)}+8t_{(1)}t_{(3)}t_{(-4)}^2+8t_{(-4)}t_{(1)}t_{(2)}^2& -8t_{(3)}^2 t_{(2)}t_{(1)}+\\ 
4t_{(4)}t_{(-3)}t_{(2)}^2+t_{(-2)}t_{(-1)}t_{(2)}t_{(1)}+t_{(-3)}t_{(-4)}t_{(3)}t_{(4)}+&4t_{(-3)}t_{(-1)}t_{(3)}t_{(1)} +\\
4t_{(1)}^3 +4t_{(3)}^3+12t_{(-4)} t_{(-2)}t_{(1)}-12t_{(-4)}t_{(2)}t_{(3)}-&12t_{(1)}t_{(-1)} -12t_{(3)}t_{(-3)}\big).
\end{align*}
\end{corollary}
\begin{proof}
We work out $P$ only since the computation for $Q$ is established in the same way but longer.  Indeed,
\begin{align*}
\mathbb{S}_\frak{D}(p)=&\frac{1}{8}\big(\mathbb{S}_\frak{D}(t_{(1)}t_{(-1)}t_{(2)}t_{(-2)})
-4\mathbb{S}_\frak{D}(t_{(1)}t_{(-2)}t_{(-4)})+\\
&2\mathbb{S}_\frak{D}(t_{(1)}t_{(-1)})+2\mathbb{S}_\frak{D}(t_{(3)}t_{(-3)})\big)\\
=&\frac{1}{8}\big(8t_{(1)}t_{(-1)}t_{(2)}t_{(-2)}-4(2t_{(1)}t_{(2)}t_{(-3)}+2t_{(-1)}t_{(-2)}t_{(3)}+\\
&2t_{(1)}t_{(-2)}t_{(-4)}+2t_{(-1)}t_{(2)}t_{(4)})+2(4t_{(1)}t_{(-1)}+\\
&4t_{(2)}t_{(-2)}+4t_{(3)}t_{(-3)}+4t_{(4)}t_{(-4)})\big)\\
=&P+3.
\end{align*}
With the help of {\it Mathematica} or a tedious hand calculation, the formula for $Q$ is equally verified.
\end{proof}

In \cite{AP} an algorithm is deduced to write {\it minimal} generators for $\C[\X]$ when $\F_r$ is free of arbitrary rank.  It is the hope of the author that exploiting 
symmetries as above will simplify the calculations involved in describing the ideals for free groups of rank $3$ or more.  Consequently, this would allow for subsequent 
advances in determining the defining relations of $\X$ in general.

\chapter{Poisson Structure on $\C[\X]$}
Let $S_{n,g}$ be a compact, connected, smooth, orientable surface of genus $g$ with $n>0$ disks removed.  Its fundamental group has the following presentation:
$$\pi_1(S_{n,g},*)=\{\xt_1,\yt_1,...,\xt_g,\yt_g,\bt_1,...,\bt_n\ |\ \xt_1\yt_1\xt_1^{-1}\yt_1^{-1}\cdots \xt_g\yt_g\xt_g^{-1}\yt_g^{-1}\bt_1\cdots \bt_n=1\},$$ which is free of 
rank $r=2g+n-1$.  And so its Euler characteristic is $\chi(S_{g,n})=1-r+0=2-2g-n.$  

\begin{figure}[h]
\begin{center}
\epsfig{file=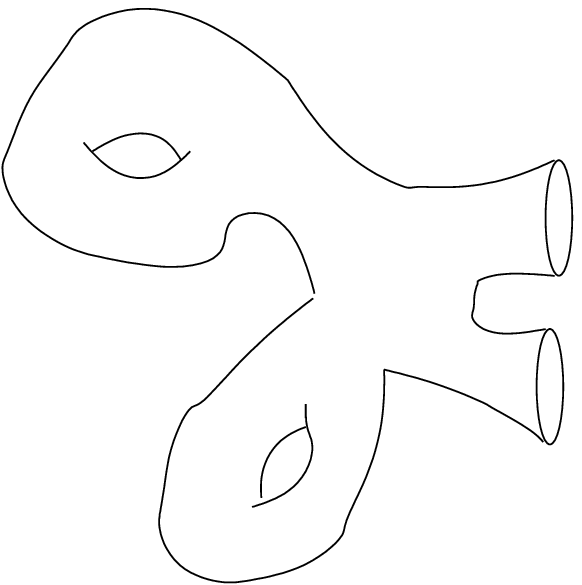}
\caption{$S_{2,2}$}
\end{center}
\end{figure}

If we assume $r>0$, then $\chi\leq 0$.  In particular, $r=1$ if and only if $\chi=0$, and in this case $S_{2,0}$ is homeomorphic 
to an annulus.  Otherwise, $\chi<0$.  The rank $r=2$ if and only if $\chi=-1$, in which case the surface is either 
$S_{3,0}$ or $S_{1,1}$; that is the three-holed sphere (or trinion or pair-of-pants), or the one-holed torus, 
respectively.

The coordinate ring of $\G\aq\G$ is $\C[\G\aq\G]=\C[\tr{\xb},\tr{\xb^{-1}}]$ which implies $\G\aq\G=\C^2$ which we parameterize by coordinates $(T_{(1)}, T_{(-1)})$.  We 
then define the boundary map $$\frak{b}_i:\X=\R\aq\G=\hm(\pi_1(S_{n,g},*),\G)\aq\G\longrightarrow\G\aq\G$$ by sending a representation class $[\rho]\mapsto 
[\rho_{|_{\bt_i}}]=(T_{(i)}, T_{(-i)})$, to the class corresponding to the restriction of $\rho$ to the boundary $\bt_i$.  This is well defined since any 
representative of $[\rho]$ is a conjugate of $\rho$, so for any $g\in \G$ the orbit of $\rho(\bt_i)$ and $g\rho(\bt_i)g^{-1}$ are identical.  We note that we are 
identifying $\X$ with conjugacy classes of representations that have closed orbits, namely those that are completely reducible, so there are no ``extended orbits'' to 
concern ourselves with.  

Subsequently, we define $\frak{b}=(\frak{b}_1,...,\frak{b}_n):\X=\G^r\aq\G\longrightarrow (\G\aq\G)^n$.  The map 
$\frak{b}$ depends on the surface, in particular, the presentation of its fundamental group.  We refer to it as a {\it peripheral structure}, and the pair 
$(\X,\frak{b})$ as the {\it relative character variety}.  Let 
$$\frak{F}=\bigcap_{i=1}^{n}\frak{b}_i^{-1}(T_{(i)}, T_{(-i)})=\{[\rho]\ |\ 
\frak{b}([\rho])=((T_{(1)}, T_{(-1)}),...,(T_{(n)}, T_{(-n)}))\}.$$  Each $\frak{F}$ is cut out of $\X$ by the 
equations $\tr{\mathbf{b}_i}=T_{(i)}$ and 
$\tr{\mathbf{b}^{-1}_i}=T_{(-i)}$ for $1\leq i\leq n$, and so is an algebraic set.  Moreover, they partition $\X$ 
since every representation has well-defined boundary values.  
 
Let $\mathcal{X}$ be the complement of the singular locus (a closed sub-variety) in $\X$, so $\mathcal{X}$ is a non-singular 
complex manifold that is dense in $\X$.  On an open dense subset of $(\G\aq\G)^n$, the Bertini theorems (see page $141$ in \cite{S}) give that 
$\frak{F}\cap\mathcal{X}$ is a non-singular submanifold of dimension $8r-8-2n=16(g-1)-6n$.  We claim that the union of these 
{\it leaves}, $\mathcal{F}=\frak{F}\cap\mathcal{X}$, foliate $\mathcal{X}$ by symplectic submanifolds, making $\mathcal{X}$ a Poisson manifold.  Moreover, the Poisson 
structure extends continuously over singularities in $\X$.  With respect to this structure, we will refer to the relative character variety $(\X,\frak{b})$ as a {\it 
Poisson variety}; that is, an affine variety whose coordinate ring is a Poisson algebra.

\section{Tangents, Cocycles, and Coboundaries}

Let $\F_r$ be a free group of rank $r$, and let $\frak{g}$ be the Lie algebra of $\G$ identified with its right invariant
vector fields.  For $\rho\in \R$, $\frak{g}$ is a $\F_r$-module, $\frak{g}_{\mathrm{Ad}_\rho}$, given by:  
$$\begin{CD}
\F_r  @>\rho>> \G @>\mathrm{Ad}>>\mathrm{Aut}(\frak{g})
\end{CD},$$ where $\mathrm{Ad}(\rho(\wt))(x)=\rho(\wt)x\rho(\wt)^{-1}$ is the adjoint representation.  

Define $\Cn{0}=\frak{g}$ and $\Cn{n}=\{\F_r^{\times n}\to \frak{g}\}$, the vector space of functions $\F_r^{\times 
n}\to \frak{g}$.  Now define $\delta_n:\Cn{n}\to\Cn{n+1}$ by
\begin{align*}
\delta_n f(\wt_1,\wt_2,...,\wt_{n+1})=&\mathrm{Ad}_\rho(\wt_1)f(\wt_2,...,\wt_{n+1})+(-1)^{n+1}f(\wt_1,...,\wt_n)+\\
&\sum_{i=1}^n(-1)^i f(\wt_1,...,\wt_i\wt_{i+1},...,\wt_{n+1}).
\end{align*}

One may verify that $\delta_{n+1}\circ\delta_{n}=0$, and so $(\Cn{*},\delta_*)$ is a cochain complex with coboundary operator 
$\delta_*$.

Let $\G^{\F_r^{\times n}}$ be the set of functions $\F_r^{\times n}\to \G$.  Let $f_t$ be a curve in $\G^{\F_r^{\times n}}$, and for 
$(\wt_1,...,\wt_n)\in\F_r^{\times n}$ let $\epsilon_{(\wt_1,...,\wt_n)}(f_t)=f_t(\wt_1,...,\wt_n)$ be the evaluation function.  Then for each evaluation, we have a 
curve in $\G$, and so we say $f_t$ is {\it smooth} if and only if it is smooth at all evaluations (see \cite{Ka}).  We define the tangent space at a function $f$ 
to be the vector space of tangents to smooth curves $f_t$ where $f_0=f$.  In other words, $$T_{f}(\G^{\F_r^{\times n}})\cong\frak{g}^{\F_r^{\times 
n}}=\Cn{n},$$ given by 
$$u_{(\wt_1,...,\wt_n)}(f)=\frac{\mathrm{d}}{\mathrm{d}t}\bigg|_{t=0}\!\!\!\!\epsilon_{(\wt_1,...,\wt_n)}(f_t)= 
\frac{\mathrm{d}}{\mathrm{d}t}\bigg|_{t=0}\!\!\!\mathrm{exp}(t\alpha_{(\wt_1,...,\wt_n)})\epsilon_{(\wt_1,...,\wt_n)}f.$$  Since 
a function is determined by its evaluations, we consider right invariant vector fields defined along these coordinates which gives a path with the requisite 
properties.

Let $I$ be a finite subset of $\F_r$, and let $\C[\G^I]$ be the coordinate ring of $\G^I$.  The set of such $I$'s is partially ordered by set inclusion and so 
$\C[\G^J]\hookrightarrow \C[\G^I]$ for $J\subset I$.  We then define $\C[\G^{\F_r}]=\underrightarrow{\lim}\C[\G^I]$.  With this said, we note that $\R$ is the subspace of
$\G^{\F_r}$ that is cut out by the functions $\mathtt{ob}_{\xt,\yt}(f)=f(\xt)f(\yt)f(\xt\yt)^{-1}$.  In other words, $$\C[\R]=\C[\G^{\F_r}]/(\mathtt{ob}_{\xt,\yt}-\id\ :\ 
\xt,\yt\in \F_r),$$ and $\R=\mathrm{Spec}_{max}(\C[\R])$.

Then $T_f(\R)=\{u\in T_f(\G^{\F_r})\ : \ u(f(\xt)f(\yt)f(\xt\yt)^{-1})=0\}$, and so 
\begin{align*}
0&=u(f(\xt)f(\yt)f(\xt\yt)^{-1})\\
&=u(f(\xt))f(\yt)f(\xt\yt)^{-1}+f(\xt)u(f(\yt))f(\xt\yt)^{-1}-f(\xt)f(\yt)f(\xt\yt)^{-1}u(f(\xt\yt))f(\xt\yt)^{-1},
\end{align*}
which implies $u_{\xt\yt}=u_\xt+\mathrm{Ad}_f(\xt)u_\yt.$
 
However, this is exactly the condition for $\delta_1(f)=0$, and so $$T_f(\R)=\mathrm{Ker}(\delta_1)=Z^1(\F_r;\frak{g}_{\mathrm{Ad}_f}).$$

On the other hand, consider a smooth path contained in the orbit $\mathcal{O}_f\subset \R$:   
$$f_t(\xt)=\mathrm{exp}(tu_\xt)f(\xt)=\mathrm{exp}(-tu_0)f(\xt)\mathrm{exp}(tu_0),$$ for
some $u_0\in \frak{g}=\Cn{0}$.  Then
$$u_\xt f(\xt)=\frac{\mathrm{d}}{\mathrm{d}t}\bigg|_{t=0}\!\!\!\!f_t(\xt)=-u_0f(\xt)+f(\xt)u_0,$$ which implies
$u_\xt=\mathrm{Ad}_f(\xt)u_0-u_0$.  However this is exactly the condition $\delta_0(u_0)=u$, so 
$$T_f(\mathcal{O}_f)=\mathrm{Image}(\delta_0)=B^1(\F_r;\frak{g}_{\mathrm{Ad}_f}).$$

Therefore, if $\rho \in \R^{reg}=\R^{s}\subset \R^{ss}$, and $[\rho]=\pi(\rho)\in \X$, then 
$$T_{[\rho]}(\X)\cong 
T_\rho(\R)/T_\rho(\mathcal{O}_\rho)=Z^1(\F_r;\frak{g}_{\mathrm{Ad}_\rho})/B^1(\F_r;\frak{g}_{\mathrm{Ad}_\rho})=H^1(\F_r;\frak{g}_{\mathrm{Ad}_\rho}).$$
 
\section{Homology and Fox Derivatives}

Let $\mathbb{Z}\F_r$ be the integral group ring of $\F_r$ and let $\epsilon:\mathbb{Z}\F_r\to \mathbb{Z}$ be the {\it 
augmentation map} defined by $\sum n_{\wt}\wt\mapsto \sum n_{\wt}$.  Define $\mathbb{Z}$-modules 
$C_0(\F_r)=\mathbb{Z}$ and $C_n(\F_r)=\mathbb{Z}\F_r^{\times n}$, and let $$\partial_{n+1}:C_{n+1}(\F_r)\longrightarrow 
C_{n}(\F_r)$$ be defined by
\begin{align*}
\partial_{n+1}(\wt_1,\wt_2,...,\wt_{n+1})=&\epsilon(\wt_1)(\wt_2,...,\wt_{n+1})+\\
&\sum_{i=1}^{n}(-1)^i(\wt_1,...,\wt_i\wt_{i+1},...,\wt_{n+1})+\\
&(-1)^{n+1}(\wt_1,...,\wt_n)\epsilon(\wt_{n+1}).
\end{align*}

One can show $\partial_{n}\circ\partial_{n+1}=0$ and so $(C_*(\F_r),\partial_*)$ is a chain complex with boundary $\partial_*$.  

In \cite{F} it is shown that the derivations on $\mathbb{Z}\F_r$, $$\mathrm{Der}(\F_r)=\{ D:\mathbb{Z} \F_r \to 
\mathbb{Z} \F_r \ : \  D( \xt \yt ) = D ( \xt ) \epsilon ( \yt ) + \xt D ( \yt ) \},$$ 
are freely generated by the derivations $\frac{\partial}{\partial\xt_i}(\xt_j)=\delta_{ij}$.  Moreover, for every 
$\mathtt{u} \in \mathbb{Z}\F_r$, there is the ``mean value 
theorem'':  
$$\mathtt{u}-\epsilon(\mathtt{u})=\sum_{i=1}^r\frac{\partial\mathtt{u}}{\partial\xt_i}(\xt_i-1).$$  These derivations 
and their generators are called the \emph{Fox derivatives}.

Recall that the peripheral structure on $\X$ is given by the presentation
$$\F_r=\pi_1(S_{n,g},*)=\{\xt_1,\yt_1,...,\xt_g,\yt_g,\bt_1,...,\bt_n\ |\ \rt=1\},$$
where $$\rt=\xt_1\yt_1\xt_1^{-1}\yt_1^{-1}\cdots \xt_g\yt_g\xt_g^{-1}\yt_g^{-1}\bt_1\cdots \bt_n=\Pi[\xt_i,\yt_i]\Pi \bt_j.$$  Then with respect to the Fox derivatives and the 
``mean value theorem,'' $\cite{Ki}$ shows 
\begin{align*}
\partial_2 \Z =&\sum_{j=1}^n \bt_j\ \mathrm{where}\\
\Z=&\sum_{i=1}^g(\frac{\partial \rt}{\partial \xt_i},\xt_i)+(\frac{\partial \rt}{\partial 
\yt_i},\yt_i)+\sum_{j=1}^n(\frac{\partial \rt}{\partial \bt_j},\bt_j).
\end{align*}
 
Consequently, he refers to $\Z$ as the \emph{fundamental relative cycle}, since $$\Z-\sum_{j=1}^n(\frac{\partial \rt}{\partial \bt_j},\bt_j)$$ is the fundamental
cycle when $n=0$, and consequently $[\Z]$ is a generator of $H_2(\F_r,\{\bt_1,...,\bt_n\};\mathbb{Z})\cong \mathbb{Z}$ (see \cite{GHJW}).

\subsection{Parabolic Cocycles}

Let $\F^i_r\subset \F_r$ be the cyclic subgroup generated by the boundary curve $\bt_i$.  Then define 
the set of {\it parabolic cocylces}, $Z^1_{par}(\F_r;\frak{g}_{\mathrm{Ad}_\rho})\subset 
Z^1(\F_r;\frak{g}_{\mathrm{Ad}_\rho})$, by 
$f\in Z^1_{par}(\F_r;\frak{g}_{\mathrm{Ad}_\rho})$ if and only if $f_i=f|_{\F_r^i}\in 
B^1(\F^i_r,\frak{g}_{\mathrm{Ad}_\rho})$ for all $1\leq i\leq n$.

It is shown in \cite{Ki} that $$B^1(\F_r;\frak{g}_{\mathrm{Ad}_\rho})\subset 
Z^1_{par}(\F_r;\frak{g}_{\mathrm{Ad}_\rho}) \subset Z^1(\F_r;\frak{g}_{\mathrm{Ad}_\rho}).$$

So for $[\rho] \in \R^{reg}\aq \G=\X^{reg}$, we have $$H^1_{par}(\F_r;\frak{g}_{\mathrm{Ad}_\rho})=Z^1_{par}(\F_r;\frak{g}_{\mathrm{Ad}_\rho}) 
/B^1(\F_r;\frak{g}_{\mathrm{Ad}_\rho})\ \subset T_{[\rho]}(\X).$$

In other words, $H^1_{par}(\F_r;\frak{g}_{\mathrm{Ad}_\rho})$ is the set of tangents that are zero on the boundary; 
that is, the tangents to representations with constant boundary value.  So the distribution 
$H^1_{par}(\F_r;\frak{g}_{\mathrm{Ad}_\rho})\subset T_{[\rho]}(\X^{reg})$
is exactly the tangents to curves in $\frak{F}\cap \X^{reg}$.  

On the other hand, let $\rho\mapsto \tr{\rho(\wt)}$ for a fixed word $\wt$ be a ``word map.''  The image of 
sufficiently many such maps (necessarily finite) determines $[\rho]$.  Then holding the 
boundary values fixed gives functions in $\C[\frak{F}]\subset \C[\X]$.  As the boundary values are deformed to a different leaf, the word map is likewise deformed.  
Therefore, the word maps generate a family of smooth invariant vector fields that generate the distribution.  

Consequently, the Stephan-Sussmann theorem (see page $17$ 
in \cite{DZ}) implies that $\X^{reg}$ is foliated by $\frak{F}\cap\X^{reg}$.  But since $\X^{reg}$ is an open dense set and all vector 
fields corresponding to the distribution are continuous (in fact, non-singular) they can be extended to a non-singular foliation on
$\mathcal{X}$, with leaves given by $\mathcal{F}=\mathcal{X}\cap \frak{F}$.  Subsequently it can be further extended 
to a singular foliation on $\X$ with leaves $\frak{F}$.

\subsection{Cup and Cap Products}
Given $u\in Z^1_{par}(\F_r,\frak{g}_{\mathrm{Ad}_\rho})$, there exists $u_0(\bt_i)\in C^0(\F_r^i,\frak{g}_{\mathrm{Ad}_\rho})$ so
$$u(\bt)=\delta_0(u_0(\bt_i))(\bt)=\mathrm{Ad}_\rho(\bt)u_0(\bt_i)-u_0(\bt_i)$$ as long as $\bt \in \F_r^i$.  Define $u_0(\wt)\in 
C^1(\F_r,\frak{g}_{\mathrm{Ad}_\rho})$ by 
setting $u_0(\wt)=0$ unless $\wt=\bt_i$, for any $1\leq i\leq n$.  In these cases, let $u_0(\bt_i)$ be a solution on $\F^i_r$ to   
$$\delta_0 \left(u_0\right)=u.$$  Although $u_0(\wt)$ is not unique, \cite{Ki} shows that if $\overline{u}_0$ is 
another such 
solution and $v\in Z^1_{par}(\F_r,\frak{g}_{\mathrm{Ad}_\rho})$ then:  $$\tr{u_0(\yt) v(\yt)}=\tr{\overline{u}_0(\yt) v(\yt)}.$$  Hence we define 
$\cup:Z^1_{par}(\F_r,\frak{g}_{\mathrm{Ad}_\rho})^{\times 2}\longrightarrow 
\mathrm{Hom}(C_2(\F_r),\C)$ 
by $$u\cup v\left(\sum n_{(\xt,\yt)}(\xt,\yt)\right)=\sum n_{(\xt,\yt)}\bigg(\mathrm{tr}\big( 
u(\xt)\mathrm{Ad}_\rho(\xt)v(\yt)\big)-\tr{u_0(\yt) v(\yt)}\bigg).$$ 
Subsequently, we define $\omega:Z^1_{par}(\F_r,\frak{g}_{\mathrm{Ad}_\rho})^{\times 2}\longrightarrow \C,$
by $$\omega(u,v)=\left(u\cup v\right)\cap \Z=u\cup v\ (\Z).$$

In \cite{Ki}, it is shown that $\omega$ is well-defined on $H^1_{par}(\F_r,\frak{g}_{\mathrm{Ad}_\rho})^{\times 2}$.
It is clear that $\omega$ is bilinear; and that for all vector fields $\xi_1$ and $\xi_2$ taking values in $H^1_{par}(\F_r,\frak{g}_{\mathrm{Ad}_\rho})$,
$$\omega(\xi_1,\xi_2): \X^{reg}\to H^1_{par}(\F_r,\frak{g}_{\mathrm{Ad}_\rho})^{\times 2}\to \C,$$ is an element of $\C[\X]$ and thus smooth.

Moreover, \cite{Ki} shows that on $H^1_{par}(\F_r,\frak{g}_{\mathrm{Ad}_\rho})^{\times 2}$, $\omega$ is skew-symmetric and non-degenerate. 
Thus $\omega$ is a smooth non-degenerate $2$-form on $\frak{F}\cap\X^{reg}$.  Thus it makes sense to ask whether it is 
{\it closed}, which \cite{Ki} shows to be true.  

In \cite{GHJW}, it is shown that $\omega$ arises from the following commutative diagram:
$$ \xymatrix{
H^1(S_{n,g},\partial S_{n,g};\gAd) \times H^1(S_{n,g};\gAd) \ar[r]^-\cup   & H^2(S_{n,g},\partial S_{n,g};\gAd \otimes \gAd) \ar[d]^{\mathrm{tr}_*}\\
    &  H^2(S_{n,g},\partial S_{n,g};\C) \ar[d]^{\cap [\Z]}\\
H^1_{par}(S_{n,g};\gAd)\times H^1_{par}(S_{n,g};\gAd)\ar[uu] \ar[r]^-\omega  & H_0(S_{n,g};\C)=\C, }$$
and it is likewise established that $\omega$ is symplectic on $H^1_{par}(\F_r,\frak{g}_{\mathrm{Ad}_\rho})^{\times 2}$.  However, this allows one to verify that 
the Poisson formula derived in \cite{G5} for closed surfaces directly generalizes to punctured surfaces.  

We note that in the above diagram $H^1_{par}(S_{n,g};\gAd)\cong H^1_{par}(\F_r,\frak{g}_{\mathrm{Ad}_\rho})$ arises from 
$j\!\!:\! H^1(S_{n,g},\partial S_{n,g};\gAd) \to H^1(S_{n,g};\gAd)$ as $\mathrm{im}(j)\cong H^1_{par}(\F_r,\frak{g}_{\mathrm{Ad}_\rho}) \cong 
H^1(S_{n,g},\partial 
S_{n,g};\gAd)/\mathrm{ker}(j)$.

Subsequently, it follows that the smooth leaves of $\X$ are symplectic and hence $\mathcal{X}$ is a Poisson manifold.  Moreover, $\X$ 
is a Poisson variety, as defined in Chapter $4$.  The dimension of $\X$ is $8r-8=16(g-1)+8n$, and imposing boundary values provide $2n$ relations (not 
necessarily independent) in $\C[\X]$.  Therefore, the dimension of any $\frak{F}$ is greater than or equal to $16(g-1)+6n$.

\section{Poisson Structure of a Trinion}

Let $S$ be an oriented surface with boundary, $\alpha, \beta \in \pi_1(S)$ (in generic position), $\alpha\cap\beta$ the set of (transverse) double 
point intersections of $\alpha$ and $\beta$, $\epsilon(p,\alpha, \beta)$ the oriented intersection number at 
$p\in \alpha\cap\beta$, and $\alpha_p\in \pi_1(S,p)$ the curve $\alpha$ based at $p$.   Then it is shown in \cite{G5} 
that 
\begin{align}\label{bracket}
\{\mathrm{tr}(\rho(\alpha)),\mathrm{tr}(\rho(\beta))\}&=\sum_{p\in 
\alpha\cap\beta}\epsilon(p,\alpha,\beta)\big((\mathrm{tr}(\rho(\alpha_p\beta_p))-(1/3)\mathrm{tr} 
(\rho(\alpha))\mathrm{tr}(\rho(\beta))\big)
\end{align}
defines a Lie bracket on $\mathbb{C}[\frak{X}]$ that is a derivation; in other words a Poisson bracket.  Moreover, it is the bracket that corresponds to the 
symplectic form $\omega$ on the leaves $\frak{F}$.

\begin{theorem}   
Let $\X$ be the relative character variety of $S=S_{3,0}$.  Then there exists a Poisson bracket on $\C[\X]$, 
where the generators $t_{(\pm i)}$ for $1\leq i\leq 3$ are Casimirs, and that is completely determined by the 
formulae:
\begin{align}
\{t_{(4)},t_{(-4)}\}&=P-2t_{(5)},\label{t4formula}\\
\{t_{(\pm 4)},t_{(5)}\}&=\frac{t_{(5)}\{t_{(\pm 4)},P\}-\{t_{(\pm 4)},Q\}}{\{t_{(-4)},t_{(4)}\}}.\label{t5formula}
\end{align}
\end{theorem}

\begin{proof}Formula $\eqref{bracket}$ shows existence of a bracket.  Since a Poisson bracket is a bilinear, 
anti-commutative derivation, it is completely determined once it is formulated on the generators of $\C[\X]$.  

We present the fundamental group of $S=S_{3,0}$ as $$\pi_1(S)=\{\xt_1,\xt_2,\xt_3\ : \ \xt_3 
\xt_2 \xt_1=1\},$$ 
so $\xt_3=\xt_1^{-1}\xt_2^{-1}$.  Hence $\pi_1(S)$ is free of rank $2$.

\begin{figure}[h!]
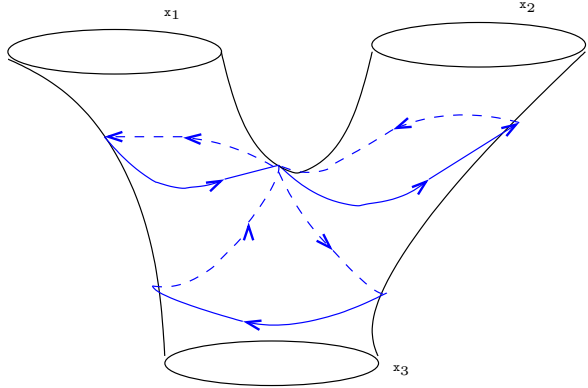

\begin{center}
\include{pantspi1}
\caption{Presentation of $\pi_1(S_{3,0},*)$}\label{pantspi1fig}
\end{center}
\end{figure}

The boundary curves in $S$ are the words $\xt_1$, $\xt_2$, and $(\xt_2\xt_1)^{-1}$, which are disjoint in the 
surface.  The sum in formula $\ref{bracket}$ is taken over intersections, and is well-defined on homotopy 
classes.  So the trace of words corresponding to disjoint curves Poisson commute; that is, they 
are Casimirs.  Hence $t_{(\pm i)}$ are Casimirs, for $1\leq i\leq 3$, since they correspond to traces of boundary curves 
(and their inverses) in $S$.  

Using the derivation property and the identity $t_{(5)}^2-Pt_{(5)}+Q=0$, we deduce:
$$t_{(5)} \{t_{(\pm 4)},P\}+P \{t_{(\pm 4)},t_{(5)}\}-\{t_{(\pm 4)},Q\}=\{t_{(\pm 4)},t_{(5)}^2\}=2t_{(5)}\{t_{(\pm 
4)},t_{(5)}\}.$$  Hence $$(2t_{(5)}-P) \{t_{(\pm 4)},t_{(5)}\}=t_{(5)} \{t_{(\pm 4)},P\}- \{t_{(\pm 4)},Q\}.$$  So $\eqref{t5formula}$ follows from $\eqref{t4formula}$.  

Now assuming $\eqref{t4formula}$ and subsequently using the explicit 
expressions of $P$ and $Q$ given in Chapter $3$, we further derive explicit expressions for $\eqref{t5formula}$ as follows.

\begin{align*}
\{t_{(4)},P\}=&(P-2t_{(5)})(t_{(4)}-t_{(1)}t_{(-2)})\\
\{t_{(4)},Q\}=&(P-2t_{(5)})(-6t_{(4)}+3t_{(-4)}^2-3t_{(-1)}t_{(-3)}-3t_{(2)}t_{(3)}+3t_{(1)}t_{(-2)}+\\
&t_{(1)}t_{(-1)}t_{(4)}+t_{(2)}t_{(-2)}t_{(4)}+t_{(3)}t_{(-3)}t_{(4)}+t_{(-1)}^2t_{(-2)}+t_{(1)}^2t_{(-3)}+\\
&t_{(2)}t_{(-3)}^2+t_{(1)}t_{(2)}^2+t_{(3)}t_{(-2)}^2+t_{(-1)}t_{(3)}^2 
+t_{(-1)}^2t_{(2)}^2-t_{(1)}t_{(-1)}t_{(2)}t_{(3)}-\\
&t_{(-3)}t_{(-2)}t_{(-1)}t_{(2)}-t_{(1)}t_{(2)}t_{(-2)}^2-t_{(-2)}t_{(-1)}t_{(1)}^2+2t_{(1)}t_{(3)}t_{(-4)}+\\
&2t_{(-2)}t_{(-3)}t_{(-4)}-4t_{(-1)}t_{(2)}t_{(-4)})
\end{align*}
\begin{align*}
\{t_{(-4)},P\}=&(2t_{(5)}-P)(t_{(-4)}-t_{(-1)}t_{(2)})\\
\{t_{(-4)},Q\}=&(2t_{(5)}-P)(-6t_{(-4)}+3t_{(4)}^2-3t_{(1)}t_{(3)}-3t_{(-2)}t_{(-3)}+3t_{(-1)}t_{(2)}+\\
&t_{(1)}t_{(-1)}t_{(-4)}+t_{(2)}t_{(-2)}t_{(-4)}+t_{(3)}t_{(-3)}t_{(-4)}+t_{(1)}^2t_{(2)}+t_{(-1)}^2t_{(3)}+\\
&t_{(-2)}t_{(3)}^2+t_{(-1)}t_{(-2)}^2+t_{(-3)}t_{(2)}^2 
+t_{(1)}t_{(-3)}^2+t_{(1)}^2t_{(-2)}^2-t_{(1)}t_{(-1)}t_{(-2)}t_{(-3)}-\\
&t_{(3)}t_{(-2)}t_{(1)}t_{(2)}  -t_{(-1)}t_{(-2)}t_{(2)}^2 -t_{(2)}t_{(1)}t_{(-1)}^2+2t_{(-1)}t_{(-3)}t_{(4)}+\\
&2t_{(2)}t_{(3)}t_{(4)}-4t_{(1)}t_{(-2)}t_{(4)}) 
\end{align*}
and so 
\begin{align*}
\{t_{(4)},t_{(5)}\}=&t_{(4)}\big(t_{(1)}t_{(-1)}+t_{(2)}t_{(-2)}+t_{(3)}t_{(-3)}-t_{(5)}-6\big)+\\
&t_{(-4)}\big(2t_{(1)}t_{(3)}+2t_{(-2)}t_{(-3)}-4t_{(-1)}t_{(2)}\big)+\\
&t_{(5)}t_{(1)}t_{(-2)}+3t_{(-4)}^2-3t_{(-1)}t_{(-3)}-3t_{(2)}t_{(3)}+3t_{(1)}t_{(-2)}+\\
&t_{(-1)}^2t_{(-2)} +t_{(1)}^2t_{(-3)}+t_{(2)}t_{(-3)}^2 +t_{(1)}t_{(2)}^2 +t_{(3)}t_{(-2)}^2+\\
&t_{(-1)}t_{(3)}^2 +t_{(-1)}^2t_{(2)}^2 -t_{(1)}t_{(-1)}t_{(2)}t_{(3)}-t_{(-3)}t_{(-2)}t_{(-1)}t_{(2)} -\\
&t_{(1)}t_{(2)}t_{(-2)}^2 -t_{(-2)}t_{(-1)}t_{(1)}^2
\end{align*}
\begin{align*}
\{t_{(-4)},t_{(5)}\}=&t_{(-4)}\big(t_{(5)}-t_{(-1)}t_{(1)}-t_{(2)}t_{(-2)}-t_{(3)}t_{(-3)}+6\big)+\\
&t_{(4)}\big(4t_{(1)}t_{(-2)}-2t_{(-1)}t_{(-3)}-2t_{(2)}t_{(3)}\big)-\\
&t_{(5)}t_{(-1)}t_{(2)}-3t_{(4)}^2+3t_{(1)}t_{(3)}+3t_{(-2)}t_{(-3)}-3t_{(-1)}t_{(2)}-\\
&t_{(1)}^2t_{(2)} -t_{(-1)}^2t_{(3)}-t_{(-2)}t_{(3)}^2 -t_{(-1)}t_{(-2)}^2 -t_{(-3)}t_{(2)}^2-\\
&t_{(1)}t_{(-3)}^2 -t_{(1)}^2t_{(-2)}^2 +t_{(-1)}t_{(1)}t_{(-2)}t_{(-3)}+t_{(3)}t_{(2)}t_{(1)}t_{(-2)} +\\
&t_{(-1)}t_{(-2)}t_{(2)}^2 +t_{(2)}t_{(1)}t_{(-1)}^2.
\end{align*}

It remains to compute $ \{t_{(4)},t_{(-4)}\}$.  Following the results in \cite{G5}, we consider  
immersed closed curves freely homotopic to $\alpha=\xt_1\xt_2^{-1}$ and $\beta=\xt_2\xt_1^{-1}$ so they intersect 
transversally at double points, and only intersect at double points.  

\begin{figure}[h!]
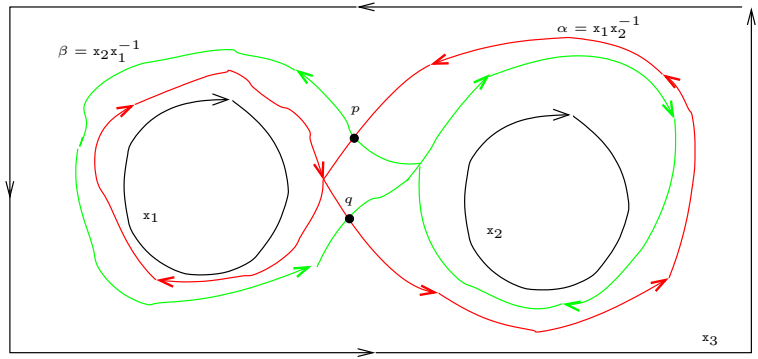

\begin{center}
\include{t4}
\caption{$\alpha$ and $\beta$ in $S$}\label{t4fig}
\end{center}
\end{figure}

Since $S$ is homotopic to a closed rectangle with two open disks removed, we depict all curves as in Figure 
$\ref{t4fig}$. 

We further let $\alpha_p$ and $\beta_p$ be the curves corresponding to $\alpha$ and $\beta$ based at the point $p$ in 
$\pi_1(S,p)$.

\begin{figure}[h!]
\begin{center}
\include{t4p}
\caption{$\alpha_p\beta_p=\xt_2^{-1}\xt_1\xt_2\xt_1^{-1}$}\label{t4pfig}
\end{center}
\end{figure}

Respectively, let $\alpha_q$ and $\beta_q$ be the corresponding curves in $\pi_1(S,q)$.

\begin{figure}[h!]
\begin{center}
\include{t4q}
\caption{$\alpha_q\beta_q=\xt_1\xt_2^{-1}\xt_1^{-1}\xt_2$}\label{t4qfig}
\end{center}
\end{figure}

Calculating the oriented intersection number at $p$ and $q$ we find $\epsilon(p,\alpha,\beta)=-1$ and  
$\epsilon(q,\alpha,\beta)=1.$

\begin{figure}[h!]
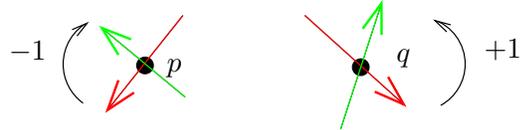

\begin{center}
\include{pqsign}
\caption{Intersection numbers at $p$ and $q$}\label{pqsignfig}
\end{center}
\end{figure}

Hence formula $\eqref{bracket}$ and Figures $\ref{t4pfig}, \ref{t4qfig},$ and $\ref{pqsignfig}$ give
\begin{align*}
\{t_{(4)},t_{(-4)}\}=&\{\mathrm{tr}(\rho(\alpha)),\mathrm{tr}(\rho(\beta))\}\\
=&\epsilon(p,\alpha,\beta)\big(\mathrm{tr}(\rho(\alpha_p\beta_p))-(1/3)\mathrm{tr}
(\rho(\alpha))\mathrm{tr}(\rho(\beta))\big)+\\
&\epsilon(q,\alpha,\beta)\big(\mathrm{tr}(\rho(\alpha_q\beta_q))-(1/3)\mathrm{tr}
(\rho(\alpha))\mathrm{tr}(\rho(\beta))\big)\\
=&-\mathrm{tr}(\rho(\alpha_p\beta_p))+\mathrm{tr}(\rho(\alpha_q\beta_q))\\
=&-\mathrm{tr}(\xb_2^{-1}\xb_1\xb_2\xb_1^{-1})+\mathrm{tr}(\xb_1\xb_2^{-1}\xb_1^{-1}\xb_2)\\
=&-t_{(5)}+t_{(-5)}\\
=&-t_{(5)}+(P-t_{(5)})=P-2t_{(5)}.
\end{align*}\qedhere\end{proof}

\subsubsection{Comment 1}
Formula $\eqref{t5formula}$ can be derived in the same manner as we derived formula 
$\eqref{t4formula}$.  Doing so leads to the expression:
$$\mathrm{tr}(\xb_1\xb_2^{-1}\xb_1^{-1}\xb_2^{-1}\xb_1\xb_2)-
\mathrm{tr}(\xb_1\xb_2^{-1})+\mathrm{tr}(\xb_2^{-2}\xb_1^{2}\xb_2\xb_1^{-1})-
\mathrm{tr}(\xb_2^{-1}\xb_1\xb_2^{-1}\xb_1\xb_2\xb_1^{-1}).$$
Subsequently using the polynomial relations derived in Chapter $2$ to reduce these trace expressions to polynomials in 
$t_{(i)}$ for $1\leq |i|\leq 5$ has provided us with verification of $\eqref{t5formula}$. 

\subsubsection{Comment 2}
For any Poisson bracket, there is a bivector field whose restriction to symplectic leaves gives a symplectic form.  Putting our calculations together and observing the 
symmetry between $\{\ti{4},\ti{5}\}$ and $\{\ti{-4},\ti{5}\}$ allows for a succint expression of the Poisson bivector field in this case.

Let $\frak{a}_{4,5}=\{\ti{4},\ti{5}\}$.  Then the Poisson bivector field, is given by: 
$$(P-2\ti{5})\frac{\partial}{\partial \ti{4}}\wedge\frac{\partial}{\partial \ti{-4}}+\frak{a}_{4,5}\frac{\partial}{\partial \ti{4}}\wedge\frac{\partial}{\partial 
\ti{5}}-\frak{i}(\frak{a}_{4,5})\frac{\partial}{\partial \ti{-4}}\wedge\frac{\partial}{\partial \ti{5}}.$$

\chapter{$\mathbb{RP}^2$-Structures on a Pair-of-Pants}

\section{Flat $\G$-Bundles and $\mathbb{RP}^2$-Structures}
Recall that the conjugation classes of representations in $\R$ whose $\G$-orbits are closed correspond to completely reducible representations and are in bijective 
correspondence with the points of $\X$.  Moreover, we showed that the conjugation classes of irreducible representations are in bijective correspondence with 
$\X^{reg}$.  Points in an (affine) algebraic quotient that have closed orbits are called {\it semi-stable}.  The points which additionally have 
zero-dimensional isotropy are called {\it stable}.  

Every semi-stable representation defines a flat $\G$-bundle over $S_{n,g}$, when $\F_r=\pi_1(S_{n,g},*)$, whose holonomy homomorphism is completely reducible.  This 
follows since given $\rho$, $\F_r$ acts on $\G$.  Thus $$E_\rho=\left(\widetilde{S_{n,g}}\times \G \right)/\F_r\to S_{n,g}$$ is a $\G$-bundle with  
holonomy $\rho$ since the fundamental group acts properly and freely on the universal cover.  We impose the discrete topology on $\G$ so that it is necessary 
flat.

On the other hand, the holonomy of a flat $\G$-bundle over $S_{n,g}$ is a representation $\F_r\to \G$.  However, conjugating the holonomy preserves the isomorphism class of 
bundle.  In other words, $\X$ parameterizes isomorphism classes of flat $\G$-bundles over $S_{n,g}$ with completely reducible holonomy.  We call such bundles 
{\it semi-stable flat $\G$-bundles}.

The group $\G_{\mathbb{R}}=\mathrm{SL}(3,\mathbb{R})$ acts transitively on $X=\mathbb{RP}^2$.  Let $S_{n,g}$ be a surface with boundary, so $\chi<0$.  An 
$(X,\G_{\mathbb{R}})$-atlas is an open cover of $S_{n,g}$ with charts $\{\phi_\alpha:U_\alpha\to X\}$ satisfying:  if $U_\alpha \cap U_\beta$ is non-empty and connected, then 
$\phi_\alpha\circ \phi_\beta^{-1} \in \G_{\mathbb{R}}$.  A maximal $(X,\G_{\mathbb{R}})$-atlas is called an $(X,\G_{\mathbb{R}})$-structure on $S_{n,g}$.  Any 
$(X,\G_{\mathbb{R}})$-manifold has a canonical bundle with fiber $X$ and discrete (coordinate changes are locally constant) structure group $\G_{\mathbb{R}}$, given by piecing 
together $U_\alpha\times X$ by coordinate changes.  Such a bundle is called a flat $(X,\G_{\mathbb{R}})$-bundle.  

We say that an $(X,\G_{\mathbb{R}})$-manifold is {\it convex} if every path is homotopic to a geodesic.  Moreover, we require that boundary components correspond to simple, closed 
geodesics contained in a geodesically convex collar neighborhood in $\mathbb{RP}^2$ whose holonomy has real, distinct, positive eigenvalues.

Under these assumptions, the $(X,\G_{\mathbb{R}})$-bundle has irreducible holonomy.  We say two such structures (with respect to isotopic diffeomorphism classes of the surface) 
are equivalent if they give rise to isomorphic bundles.  It can be thus shown that the moduli space of such 
structures embeds in $\X$ by mapping the structure to the conjugacy class of its holonomy homomorphism (see \cite{G1,G2}).  

\section{Fibration of Convex $\mathbb{RP}^2$-Structures}

Let $\frak{P}(S_{n,g})$ be the moduli space of convex $\mathbb{RP}^2$ structures on $S_{n,g}$, and let $\frak{P}(\partial S_{n,g})$ be the space of germs of 
convex projective 
structures on collar neighborhoods of the boundaries.  Then there is a map 
\begin{equation}\label{fibration}
\frak{P}(S_{n,g})\longrightarrow \frak{P}(\partial S_{n,g}),
\end{equation}
given by restricting the holonomy homomorphism to the boundary components.    

Define the {\it discriminant} $$\mathtt{d}(x,y)=x^2y^2-4(x^3+y^3)+18xy-27,$$ and note that $\mathtt{d}$ is zero if 
and only if there is a repeated root of the characteristic
polynomial $t^3-xt^2+yt-1$.  Let $\frak{P}(\bt_i)\subset \G_{\mathbb{R}}\aq \G_{\mathbb{R}}$ be defined by our boundary condition; that is, conjugacy classes of matrices with real 
distinct positive eigenvalues.  Then for each $1\leq i\leq n$, $\frak{P}(\bt_i)$ is determined by $x=\tr{\mathbf{b}_i}>0$ and $y=\tr{\mathbf{b}_i^{-1}}>0$ and $\mathtt{d}(x,y)>0$. 
It is shown in \cite{G2} that $\frak{P}(\partial S_{n,g})\cong \frak{P}(\bt_1)\times \cdots \times \frak{P}(\bt_n)\cong \mathbb{R}^{2n},$ and $\eqref{fibration}$ is a fibration.

So the foliation $\X\longrightarrow (\G\aq\G)^{\times n}$ restricts to a fibration on the image $\frak{P}(S_{n,g})\to \X$.

A key step in Goldman's proof of $\eqref{fibration}$ is establishing it for the case when the surface is a three-holed 
sphere, or a trinion.  In this case, explicit forms for the boundary matrices are formulated.  Using the resulting 
equations we prove the following theorem.

\begin{theorem}  
Let $\frak{P}$ be the image in $\X$ of the moduli space of convex $\mathbb{RP}^2$-structures on a three-holed 
sphere.  Then $\frak{P}$ is a real $2$-dimensional fibration over a real $6$-dimensional base defined by the 
inequalities for $1\leq i\leq 3$: $t_{(\pm i)}>0$ and $\mathtt{d}(t_{(i)},t_{(-i)})>0,$
and the fiber is determined by expressions for $t_{(\pm 4)}$ in terms of the other six generators and two free positive parameters $s,t$.
\end{theorem}

\begin{proof}
The foliation map $\frak{b}$ restricts to $\frak{P}$ and so provides the fibration with stated dimensions, since our preceding remarks imply that the following diagram
commutes:  
\begin{displaymath}
\begin{CD}
\X@>\frak{b}>>(\G\aq\G)^{\times n}\\
@AAA      @AAA\\
\frak{P}(S_{3,0}) @>>> \frak{P}(\partial S_{3,0}).\\
\end{CD}
\end{displaymath}

The restriction map is defined by $\tr{\mathbf{x}_i}>0$, $\tr{\mathbf{x}_i^{-1}}>0$ and $\mathtt{d}(\tr{\mathbf{x}_i},\tr{\mathbf{x}_i^{-1}})>0$ for $1\leq i\leq 3$ since these 
are the boundary generators.

However, in the coordinate ring $\C[\X]$, for $1\leq i\leq 3$, this corresponds to:
\begin{align*}
&t_{(\pm i)}>0\ \ \text{and}\ \ \mathtt{d}(t_{(i)},t_{(-i)})>0.
\end{align*}

Since $\ti{5}$ is locally determined by the other variables using $\ti{5}^2-P\ti{5}+Q=0$, and the Casimirs are fixed in a given fiber of the 
boundary map, we are left with only $\ti{\pm 4}$ to determine.  Using {\it Mathematica} we verify that the fiber is given
by explicit equations for these generators in terms of the two positive free parameters $s,t$ given in \cite{G2}.

Let $\lambda_1,\lambda_2,$ and $\lambda_3$ be the largest eigenvalue of $\xb_1, \xb_2$ and $\xb_3$ respectively.  Then:
\begin{align*}
&\lambda_1^3 + \ti{1}\lambda_1^2 - \ti{-1}\lambda_1+ 1 = 0\\
&\lambda_2^3 + \ti{2}\lambda_2^2 - \ti{-2}\lambda_2+ 1 = 0 \\
&\lambda_3^3 + \ti{3}\lambda_3^2 - \ti{-3}\lambda_3+ 1 = 0, 
\end{align*}
and $\lambda_1,\lambda_2,$ and $\lambda_3$ are locally expressed in terms of the Casimirs.  Then using the expressions 
for the boundary matrices given in \cite{G2}, we derive the following formulas for $t_{(\pm 4)}$:

\begin{align*}
\ti{4}=&\frac{1}{s \sqrt{\lambda_1} \sqrt{\lambda_2} \sqrt{\lambda_3} } + \frac{\sqrt{\lambda_1} \sqrt{\lambda_2} \sqrt{\lambda_3} }{s} -
\frac{s \lambda_1^{3/2} \sqrt{\lambda_2} }{\sqrt{\lambda_3}} - \frac{s \lambda_2^{3/2} \sqrt{\lambda_3} }{\sqrt{\lambda_1}}- \\
&\frac{s \sqrt{\lambda_1} \lambda_3^{3/2} }{\sqrt{\lambda_2}} + 2s^2 - \frac{\lambda_1}{t \lambda_2 } - 
\frac{\lambda_3}{t \lambda_1 } + \frac{1}{st \sqrt{\lambda_1} \sqrt{\lambda_2} \sqrt{\lambda_3} } + \\
&\frac{s \sqrt{\lambda_2} }{t \lambda_1^{3/2} \sqrt{\lambda_3} } + \frac{s \sqrt{\lambda_3} }{t \sqrt{\lambda_1} \lambda_2^{3/2} } + 
\frac{s \sqrt{\lambda_1} \lambda_3^{3/2} }{t \sqrt{\lambda_2} } - \frac{s^2}{t} - \frac{s^2 \lambda_3^2 }{t \lambda_1 \lambda_2 } +\\ 
&\frac{s^3 \sqrt{\lambda_3}}{t \lambda_1^{3/2} \sqrt{\lambda_2} } - t \lambda_1 \lambda_2^2  + 
\frac{t \sqrt{\lambda_1} \sqrt{\lambda_2} \sqrt{\lambda_3} }{s} + \frac{st \lambda_1^{3/2} \sqrt{\lambda_2}}{\sqrt{\lambda_3}} +\\ 
&\frac{ \ti{1} }{ \lambda_2 } - \lambda_2^2 \ti{1} + \frac{s \sqrt{\lambda_1} \sqrt{\lambda_2} \ti{1} }{ \sqrt{\lambda_3} } +
\frac{ \ti{1} }{ t \lambda_2 } - \frac{s \lambda_3^{3/2}  \ti{1} }{t \sqrt{\lambda_1} \sqrt{\lambda_2}  } + \\
&\frac{ s^2 \ti{1} }{ t \lambda_1 } + \frac{s \sqrt{\lambda_2} \sqrt{\lambda_3}  \ti{2} }{ \sqrt{\lambda_1} } + 
t \lambda_1 \lambda_2 \ti{2} + \lambda_2 \ti{1} \ti{2} +\frac{ s \sqrt{ \lambda_1} \sqrt{\lambda_3}  \ti{-3} }{ \sqrt{\lambda_2}} +\\ 
&\frac{\ti{-3}}{t \lambda_1 } - \frac{s \sqrt{\lambda_1} \sqrt{\lambda_3} \ti{-3}}{t \sqrt{\lambda_2} } + 
\frac{s^2 \lambda_3 \ti{-3}}{t \lambda_1 \lambda_2 } + \frac{s \sqrt{\lambda_3} \ti{1} \ti{-3}}{t \sqrt{\lambda_1}\sqrt{\lambda_2} }
\end{align*}

\begin{align*}
\ti{-4}=&\frac{2}{s^2} - \frac{\sqrt{\lambda_1}\lambda_2^{3/2}}{s \sqrt{\lambda_3} } - \frac{\lambda_1^{3/2}\sqrt{\lambda_3}}{s \sqrt{\lambda_2}}-
\frac{\sqrt{\lambda_2}\lambda_3^{3/2}}{s \sqrt{\lambda_1}} + \frac{s}{\sqrt{\lambda_1} \sqrt{\lambda_2} \sqrt{\lambda_3} } +s\sqrt{\lambda_1}\sqrt{\lambda_2}\sqrt{\lambda_3}+\\ 
&\frac{ \lambda_2 }{ t \lambda_1} + \frac{ \lambda_3 }{t \lambda_2 } + \frac{\lambda_1 \lambda_3^2 }{t} + \frac{1}{s^2 t} - 
\frac{ \lambda_1^{3/2} \sqrt{\lambda_3}}{st \sqrt{\lambda_2}}- \frac{ \sqrt{ \lambda_2 } \lambda_3^{3/2} }{st \sqrt{\lambda_1}} -\\
&\frac{s \sqrt{\lambda_1} \sqrt{\lambda_2} \sqrt{\lambda_3}}{t} - \frac{s \lambda_3^{5/2}}{t \sqrt{\lambda_1} \sqrt{\lambda_2}} + 
\frac{s^2 \lambda_3}{t \lambda_1} + \frac{t \lambda_1}{\lambda_3} + \frac{t}{s^2} - \frac{t \sqrt{\lambda_1} \lambda_2^{3/2}}{s \sqrt{\lambda_3}} +\\ 
&\frac{\sqrt{\lambda_1} \sqrt{\lambda_3} \ti{1}}{s \sqrt{\lambda_2}} - \frac{\lambda_3^2 \ti{1}}{t} + \frac{\sqrt{\lambda_1} \sqrt{\lambda_3} \ti{1}}{st \sqrt{\lambda_2}} + 
\frac{s \sqrt{\lambda_2} \sqrt{\lambda_3} \ti{1}}{t\sqrt{\lambda_1}} + \frac{\ti{2}}{\lambda_1} - \\
&\lambda_1^2 \ti{2} + \frac{\sqrt{\lambda_1} \sqrt{\lambda_2} \ti{2}}{s \sqrt{\lambda_3}} + \frac{t \sqrt{\lambda_1} \sqrt{\lambda_2} \ti{2}}{s \sqrt{\lambda_3}} + 
\lambda_1 \ti{1} \ti{2} + \frac{\sqrt{\lambda_2}\sqrt{\lambda_3} \ti{-3}}{s \sqrt{\lambda_1}} -\\ 
&\frac{\lambda_1 \lambda_3 \ti{-3}}{t} + \frac{\sqrt{\lambda_2}\sqrt{\lambda_3} \ti{-3}}{st \sqrt{\lambda_1}} +
\frac{s \lambda_3^{3/2} \ti{-3}}{\sqrt{t \lambda_1} \sqrt{\lambda_2}} + \frac{\lambda_3 \ti{1} \ti{-3}}{t}
\end{align*}

Since they are in terms of only the Casimirs and $s,t$, and any conjugacy class of a representation is determined by its values on $\ti{\pm i}$ for $1\leq i\leq 
4$ and $\ti{5}$, it follows that the fibers are completely determined by these equations.
\end{proof}


\begin{thebibliography}{100}

\bibitem[AP]{AP} Abeasis, A., and Pittaluga, M., {\it On a minimal set of generators for the invariants of
$3\times 3$ matrices}, Comm. Alg. $\mathbf{17(2)}$ ($1989$), $487$-$499$

\bibitem[ADS]{ADS} Aslaksen, H., Drensky, V., and Sadikova, L., {\it Defining relations of invariants of two $3\times 3$ matrices},
arXiv:math.RA/0405389 v1, $2004$

\bibitem[A]{A} Artin, M., {\it On Azumaya algebras and finite dimensional representations of rings}, J. of Alg., $\mathbf{11}$ ($1969$), $532$-$563$

\bibitem[B]{B} Brown, K., ``Cohomology of groups,''  Graduate Texts in Mathematics No. $87$, Spring-Verlag New York, $1982$

\bibitem[CLO]{CLO} Cox, D., Little, J., and O'Shea, D., ``Using Algebraic Geometry,''  Graduate   
Texts in Mathematics No. $185$, Spring-Verlag New York, $1998$

\bibitem[CSM]{CSM} Carter, R., Segal, G., and Maconald, I., ``Lectures on Lie Groups and Lie Algebras,'' London
Mathematical Society Student Texts No. $32$, Cambridge University Press, Cambridge, $1995$

\bibitem[D]{D} Dolgachev, I.,``Lectures on Invariant Theory,'' London Mathematical Lecture Notes Series $296$,
Cambridge University Press, $2003$

\bibitem[DF]{DF} Drensky, V., and Formanek, E., ``Polynomial Identity Rings,'' Advanced Courses in Mathematics CRM
Barcelona, Birkh$\ddot{\mathrm{a}}$user Verlag Basel, $2004$

\bibitem[Du]{Du} Dubnov, J., Sur une g\'{e}n\'{e}ralisation de l'\'{e}quation de Hamilton-Caley et sur les invariants simultan\'{e}s de plusieurs affineurs,  
Proc. Seminar on Vector and Tensor Analysis, Mechanics Research Inst., Moscow State Univ. $\mathbf{2/3}$ ($1935$), $351$-$367$.

\bibitem[DZ]{DZ} Dufour, J., Zung, N.T., ``Poisson Structures and their Normal Forms,''  Progress in Mathematics, Vol. $242$, Birkh$\ddot{\mathrm{a}}$user Verlag Berlin, $200$. 

\bibitem[E]{E} Eisenbud, D.,  ``Commutative Algebra with a View Toward Algebraic Geometry,''  Graduate Texts in Mathematics No.
150, Spring-Verlag New York, 1995

\bibitem[F]{F} Fox, R., {\it Free differential calculus. I}, Ann. of Math., $\mathbf{57}$ ($1953$), $547$-$560$

\bibitem[H]{H} Hungerford, T., ``Algebra,'' Graduate Texts in Mathematics No. $73$, Spring-Verlag New York, $1974$

\bibitem[G1]{G1} Goldman, W., {\it Geometric structures on manifolds and varieties of representations}, Contemp.
Math. $\mathbf{74}$ ($1988$), $169$-$197$

\bibitem[G2]{G2} Goldman, W., {\it Convex real projective structures on compact surfaces}, J. Diff. Geo. $\mathbf{31}$ ($1990$), $791$-$845$

\bibitem[G3]{G3} Goldman, W., {\it Introduction to character varieties}, unpublished notes ($2003$)

\bibitem[G4]{G4} Goldman, W., {\it The Symplectic Nature of Fundamental Groups of Surfaces}, Advances in Math.
$\mathbf{54}$ ($1984$), $200$-$225$

\bibitem[G5]{G5} Goldman, W., {\it Invariant functions on Lie groups and Hamiltonian flows of surface group representations}, Invent. Math. $\mathbf{85}$ ($1986$), 
$263$-$302$

\bibitem[GHJW]{GHJW} Guruprasad, K., Huebschmann, J., Jeffrey, L., Weinstein, A., {\it Group systems, groupoids, and moduli spaces of parabolic bundles}, Duke Math. J. 
$\mathbf{89}$ ($1997$), no. $2$, $377$-$412$

\bibitem[Ka]{Ka} Karshon, Y., {\it An algebraic proof for the symplectic structure of moduli space}, Proc. Amer. Math. Soc. $\mathbf{116}$ ($1992$), $591$-$605$

\bibitem[Ki]{Ki} Kim, H., {\it The symplectic global coordinates on the moduli space of real projective structures}, J. Diff. Geo. $\mathbf{53}$ ($1999$), $359$-$401$

\bibitem[L]{L} Lawton, S., {\it Relations and Symmetries of $\mathrm{SL}(3,\C)^2/\!\!/\mathrm{SL}(3,\C)$}, \texttt{arxiv.org/ math.AG/$0601132.$}

\bibitem[LP]{LP} Lawton, S. and Peterson, E., {\it Spin Networks and $\mathrm{SL}(2,\C)$-Character Varieties}, \texttt{arxiv.org/math.QA/$0511271$.}

\bibitem[MS]{MS} Marincuk, A. and Sibirskii, K., {\it Minimal polynomial bases of affine invariants of square
matrices of order three}, Mat. Issled. $\mathbf{6}$ ($1971$), $100$-$113$

\bibitem[MKS]{MKS} Magnus, W., Karrass, A., and Solitar, D., ``Combinatorial Group Theory: Presentations of Groups in Terms of
Generators and Relations,'' Pure and Applied Mathematics Vol. XIII, Interscience Publishers, $1966.$

\bibitem[N]{N} Nakamoto, K., {\it The structure of the invariant ring of two matrices of degree $3$}, J. Pure and
Applied Alg. $\mathbf{166}$ ($2002$), $125$-$148$

\bibitem[R]{R} Razmyslov, Y., {\it Trace identities of full matrix algebras over a field of characteristic zero.} (Russian)
Izv. Akad. Nauk SSSR Ser. Mat.  $\mathbf{38}$  $(1974)$, $723$-$756$.

\bibitem[PX]{PX} Previte, J. and Xia, E., Various letters to the author.

\bibitem[P1]{P1} Procesi, C., {\it Invariant theory of $n\times n$ matrices}, Advances in Mathematics {\bf $19$}   
($1976$), $306$-$381$

\bibitem[P2]{P2} Procesi, C., {\it Finite dimensional representations of algebras}, Israel J. Math. {\bf $19$}
($1974$), $169$-$182$

\bibitem[S]{S} Shafarevich, I., ``Basic Algebraic Geometry $1$,'' $2^{\text{nd}}$ edition,  Springer-Verlag, Berlin, $1994$

\bibitem[Si]{Si} Sikora, A., {\it $SL_n$-Character Varieties as Space of Graphs}, Trans. Amer. Math. Soc. $\mathbf{353}$  $(2001)$,  no. $7$, $2773$-$2804$

\bibitem[SR]{SR} Spencer,A. and Rivlin, R., {\it Further results in the theory of matrix polynomials}, Arch. Rational Mech. Anal. $\mathbf{4}$
($1960$), $214$-$230$

\bibitem[T]{T} Teranishi, Y., {\it The ring of invariants of matrices},  Nagoya Math. J. $\mathbf{104}$  $(1986)$, $149$-$161$.

\end{thebibliography}
\end{document}